\documentclass[final,onefignum,onetabnum]{siamltex1213}
\usepackage{amssymb}
\newtheorem{assump}{Assumption}
\newtheorem{remark}{Remark}
\newtheorem{problem}{Problem}

\usepackage{mathtools}
\usepackage[normalem]{ulem} 

\newcommand{\vf}{{\mathbf{f}}}
\newcommand{\vg}{{\mathbf{g}}}

\newcommand{\vv}{{\mathbf{v}}}
\newcommand{\vw}{{\mathbf{w}}}
\newcommand{\vx}{{\mathbf{x}}}
\newcommand{\vy}{{\mathbf{y}}}
\newcommand{\vz}{{\mathbf{z}}}

\newcommand{\vB}{{\mathbf{B}}}
\newcommand{\vC}{{\mathbf{C}}}

\newcommand{\vH}{{\mathbf{H}}}

\newcommand{\vN}{{\mathbf{N}}}

\newcommand{\vR}{{\mathbf{R}}}
\newcommand{\vS}{{\mathbf{S}}}

\newcommand{\vV}{{\mathbf{V}}}
\newcommand{\vW}{{\mathbf{W}}}

\newcommand{\vZ}{{\mathbf{Z}}}

\newcommand{\cB}{{\mathcal{B}}}

\newcommand{\cG}{{\mathcal{G}}}
\newcommand{\cH}{{\mathcal{H}}}

\newcommand{\cS}{{\mathcal{S}}}
\newcommand{\cT}{{\mathcal{T}}}

\newcommand{\dom}{{\mathrm{dom}}} 

\newcommand{\prox}{\mathbf{prox}}

\newcommand{\tnabla}{\widetilde{\nabla}}

\newcommand{\TPRS}{T_{\mathrm{PRS}}}

\DeclareMathOperator*{\argmin}{arg\,min}

\DeclareMathOperator*{\Min}{minimize}

\DeclareMathOperator*{\zer}{zer}

\DeclareMathOperator*{\gra}{gra}

\newcommand{\bc}{\begin{center}}
\newcommand{\ec}{\end{center}}

\newcommand{\bdm}{\begin{displaymath}}
\newcommand{\edm}{\end{displaymath}}

\newcommand{\beq}{\begin{equation}}
\newcommand{\eeq}{\end{equation}}

\newcommand{\bfl}{\begin{flushleft}}
\newcommand{\efl}{\end{flushleft}}

\newcommand{\bt}{\begin{tabbing}}
\newcommand{\et}{\end{tabbing}}

\newcommand{\beqn}{\begin{align}}
\newcommand{\eeqn}{\end{align}}

\newcommand{\beqs}{\begin{align*}} 
\newcommand{\eeqs}{\end{align*}}  

\usepackage[lined,boxed,commentsnumbered, ruled,vlined]{algorithm2e}

\newcommand\numberthis{\addtocounter{equation}{1}\tag{\theequation}}

\DeclarePairedDelimiter{\dotp}{\langle}{\rangle}

\newcommand{\refl}{\mathbf{refl}}

\usepackage{booktabs}

\def\cut#1{{}}

\title{Convergence rate analysis of primal-dual splitting schemes\thanks{This work is partially supported by NSF GRFP grant DGE-0707424. 
}}

\author{Damek Davis\thanks{Department of Mathematics, University of California, Los Angeles,
              Los Angeles, CA 90025/
              School of Operations Research and Information Engineering, Cornell University, 
              Ithaca, NY 14850
              \email{(damek@math.ucla.edu)}}}

\begin{document}
\maketitle
\slugger{siopt}{xxxx}{xx}{x}{x--x}

\begin{abstract}
Primal-dual splitting schemes are a class of powerful algorithms that solve complicated monotone inclusions and convex optimization problems that are built from many simpler pieces. They decompose problems that are built from sums, linear compositions, and infimal convolutions of simple functions so that each simple term is processed individually via proximal mappings, gradient mappings, and multiplications by the linear maps. This leads to easily implementable and highly parallelizable or distributed algorithms, which often obtain nearly state-of-the-art performance.

In this paper, we analyze a monotone inclusion problem that captures a large class of primal-dual splittings as a special case.  We introduce a unifying scheme and use some abstract analysis of the algorithm to prove convergence rates of the proximal point algorithm, forward-backward splitting, Peaceman-Rachford splitting, and forward-backward-forward splitting applied to the model problem.  Our ergodic convergence rates are deduced under variable metrics, stepsizes, and relaxation.  Our nonergodic convergence rates are the first shown in the literature. Finally, we apply our results to a large class of primal-dual algorithms that are a special case of our scheme and deduce their convergence rates. 
\end{abstract}



\pagestyle{myheadings}
\thispagestyle{plain}
\markboth{D. Davis}{Convergence rates in primal-dual splitting schemes}

\section{Introduction}
Primal-dual algorithms are abstract splitting schemes that solve monotone inclusion and convex optimization problems.  These schemes fully decompose problems built from sums, linear compositions, parallel sums, and infimal convolutions of simple functions so that each simple term is processed individually.  This decomposition is achieved by cleverly combining primal and dual pair problems into a single inclusion problem, to which standard operator splitting algorithms can be applied. This process gives rise to algorithms that are inherently parallel or distributed and in which expensive matrix inversions can be avoided.  The characteristics of primal-dual algorithms are especially desirable for large-scale applications in machine learning, image processing, distributed optimization, and control.  

Primal-dual methods have a long history with many contributors, and an attempt to summarize and relate all of the contributions is beyond the scope of this paper. In this paper, we are mainly concerned with the line of work that began in \cite{pock2009algorithm,chambolle2011first,esser2010general} and the many generalizations and enhancements of the basic framework that followed \cite{combettes2012primal,condat2013primal,vu2013splitting,briceno2011monotone+,bo?2013douglas,bot2013algorithm,combettes2013systems,komodakis2014playing,bo2014convergence,combettes2014forward}.  Thus, we consider the following prototypical convex optimization problem as our guiding example:
\begin{align}\label{eq:pdsimpleproblem}
\Min_{x \in \cH_0} f(x) + g(x) + \sum_{i=1}^n (h_i \square l_i)(B_i x)
\end{align}
where $\square$ denotes the infimal convolution operation (see Section~\ref{sec:notation}), $n \in \vN$, $n \geq 1$, $\cH_i$ are Hilbert spaces for $i=0, \ldots, n$, the functions $f, g : \cH_0 \rightarrow (-\infty,\infty]$ and $h_i, l_i : \cH_i\rightarrow (-\infty, \infty]$ are closed, proper, and convex for $i=1, \cdots, n$, and $B_i : \cH_0\rightarrow \cH_i$ is a bounded linear map for $i=1, \ldots, n$.  

All of the algorithms presented in this paper completely disentangle the structure of Problem~\eqref{eq:pdsimpleproblem} so that each iteration only involves the individual proximal operators of each of the nondifferentiable terms, the gradient operators of the differentiable terms, and multiplication by the linear maps.  Thus, the maps $B_i$ are never inverted, and we never compute proximal operators or gradients of sums or infimal convolutions of functions.  We note that this level of separability is not achieved by classical splitting methods such as forward-backward splitting, Douglas-Rachford splitting, or the alternating direction method of multipliers (ADMM) when they are applied directly to the primal optimization Problem~\eqref{eq:pdsimpleproblem} \cite{bruck1977weak,passty1979ergodic,GlowinskiADMM,lions1979splitting}.   

In Problem~\eqref{eq:pdsimpleproblem}, the maps $B_i$ can be used as ``data matrices," in which case $h_i$ and $l_i$ are \emph{data fitting} terms and $f$ and $g$ enforce \emph{prior knowledge} on the structure of the solution, such as sparsity, low rank, or smoothness.  In other cases, the maps $h_i$ and $l_i$ may be regularizers that emphasize many competing structures. We now present an example.

{\bf Application: Constrained model fitting with group-structured regularizers.} Fix $d, m \in \vN \backslash \{0\}$. Suppose we are given a measurement $b \in \vR^d$ and a dictionary $A \in \vR^{d\times m}$. Our goal is to recover a highly structured signal $x = (x_1, \cdots, x_m)^T \in \vR^m$ such that $Ax \approx b$.  For example, in the hierarchical sparse coding problem (HSCP) \cite{jenatton2011proximal}, we arrange the columns of $A$ into a directed tree structure $\cT$ and allow $x_i = 0$ only if $x_j = 0$ for all descendants $j$ in $\cT$ of node $i$.  Such a hierarchical representation is particularly useful for multi-scale data such as images and text documents. This type of regularization can be generalized to include arbitrary column groupings and complicated relationships between the elements of each group. Indeed, let $G$ be a set of (possibly overlapping) subsets of $\{1, \cdots, m\}$. For all $S \in G$ and $x \in \vR^m$, let $B_Sx = L_S(x_i)_{i \in S}^T \in \vR^{m_S}$ where $m_S \in \vN \backslash \{0\}$ and $L_S : \vR^{|S|} \rightarrow \vR^{m_{S}}$ is a linear map. Let $C \subseteq \vR^m$ be a closed convex set, and let $\iota_C : \vR^m \rightarrow \{0, \infty\}$ be the convex indicator function of $C$. For all $S \in G$, let $h_S : \vR^{m_S} \rightarrow (-\infty, \infty]$ be a closed, proper, and convex regularizer, and let $l_S = \iota_{\{0\}}$, which implies $h_S \square l_S = h_S$.
Then one special case of Problem~\eqref{eq:pdsimpleproblem} is the group-structured regularized model fitting problem:
\begin{align*}
\Min_{x \in \vR^m} \; \iota_{C}(x) + (1/2)\|Ax - b\|^2 + \sum_{ S \in G} h_S(B_Sx).
\end{align*}
In \cite{jenatton2011proximal}, the authors consider the nonegativity constraint $C = \vR^m_{\geq 0}$ and a grouping $G$ which consists of overlapping sets $S_i$ for $i \in \{1, \cdots, m\}$ such that $S_i$ contains $i$ and all of the descendants of $i$ in $\cT$. Furthermore,  for each $S \in G$, they consider the map $L_S = I_{\vR^{|S|}}$ and the function $h_S = w_S\|(x_i)^T_{i \in S}\|_p$ where $p \in [1, \infty]$ and $w_S > 0$. This setup induces a mixed $\ell_1/\ell_p$ norm on $\vR^{m}$ of the form $\sum_{S \in G} w_S\|(x_i)_{i \in S}^T\|_p$, which tends to ``zero out" entire groups of components. Note that the sum is also highly nonseparable in the components of $x$, which can make the proximal operator of the regularization term difficult to evaluate.  If we denote $f(x) = \iota_C(x)$ and $g(x) = (1/2)\|Ax - b\|^2$, then the algorithms in this paper only utilize the projection $P_C = \prox_f$ onto $C$, the gradient $\nabla g(x) = A^\ast (Ax - b)$, and for all $S \in G$ in parallel, multiplications by the maps $B_S$ and $B_S^\ast$, and evaluations of the proximal operator of the function $h_S$. Not only does this make each iteration of the algorithm simple to implement and computationally inexpensive, it also provides a unified algorithmic framework for higher order regularizations of the components in each group, a task which might otherwise be intractable in large-scale applications.

Finally, we note that the use of infimal convolutions in applications is not wide-spread, so we list a few instances where they may be useful: Infimal convolutions are used in image recovery \cite[Section 5]{chambolle1997image} to remove staircasing effects in the total variation model. The infimal convolution of the indicator functions of two closed convex sets is the indicator function of their Minkowski sum, which has applications in motion planning for robotics \cite[Section 4.3.2]{lavalle2006planning}.  In convex analysis, the Moreau envelope of a function arises as an infimal convolution with a multiple of the squared norm \cite[Section 12.4]{bauschke2011convex}. More generally, the infimal convolution of $h_i$ and $l_i$ can be interpreted as a regularization or smoothing of $h_i$ by $l_i$ and vice versa \cite[Section 18.3]{bauschke2011convex}.

\subsection{Goals, challenges, and approaches}
This work seeks to improve the theoretical understanding of the convergence rates of primal-dual splitting schemes. In this paper, we study primal-dual algorithms that are applications of standard operator splitting algorithms in product spaces consisting of primal and dual variables. Consequently, the convergence theory for these algorithms is well-developed, and they are known to converge (weakly) under mild conditions.

Although we understand when these algorithms converge, relatively little is known about their rate of convergence. For convex optimization algorithms, the \emph{ergodic} convergence rate of the \emph{primal-dual gap} has been analyzed in a few cases \cite{chambolle2011first,boct2013convergence,bo2014convergence,doi:10.1137/130910774}.  However, even in cases where convergence rates are known, variable metrics and stepsizes, which can significantly improve practical performance of the algorithms \cite{pock2011diagonal,goldstein2013adaptive}, are not analyzed.  In addition, we are not aware of any convergence rate analysis of the primal-dual gap for the \emph{nonergodic} (or last) iterate generated by these algorithms. It is important to understand nonergodic convergence rates because the ergodic (or time-averaged) iterates can ``average out" structural properties, such as sparsity and low rank, that are shared by the solution and the nonergodic iterate.

The convergence rate analysis of the ergodic primal-dual gap largely follows from subgradient inequalities and an application of Jensen's inequality. In contrast, the techniques developed in this paper exploit the properties of the nonexpansive operators driving the algorithms to deduce the nonergodic convergence rate of the primal-dual gap.  Thus, our techniques are quite different from those used in classical convergence rate analysis and parallel the analysis developed in \cite{davis2014convergence}.

We summarize our contributions and techniques as follows:
\begin{romannum}
\item We describe a model monotone inclusion problem that generalizes many primal-dual formulations that appear in the literature. We provide a simple prototype algorithm to solve the model problem, and we deduce a fundamental inequality that bounds the primal-dual gap at each iteration of the algorithm. We then simplify the inequality in the special case of four splitting algorithms (Section~\ref{sec:US}).

\item We derive ergodic convergence rates of the variable metric forms of the relaxed proximal point algorithm (PPA), relaxed forward-backward splitting (FBS), and forward-backward-forward splitting as well as the fixed metric relaxed Peaceman-Rachford splitting (PRS) algorithm (Section~\ref{sec:ergodic}). After some algebraic simplifications, our analysis essentially follows from an application of Jensen's inequality.

\item We derive nonergodic convergence rates of relaxed PPA, relaxed FBS, and relaxed PRS (Section~\ref{sec:nonergodic}). All of our analysis follows by bounding the primal-dual gap function by a multiple of the \emph{fixed-point residual} (FPR) of the nonexpansive mapping that drives the algorithm.  Thus, we show that the size of the FPR can be used as a valid \emph{stopping criteria} for these three algorithms.

\item We apply our results to deduce ergodic and nonergodic convergence rates for a large class of primal-dual algorithms that have appeared in the literature (Section~\ref{sec:applications}).

\end{romannum}

Our analysis not only deduces the convergence rates of a large class of primal-dual algorithms found in the literature. It also serves as a resource for the analysis of future primal-dual algorithms that solve  generalizations of Problem~\ref{eq:pdsimpleproblem}, e.g., \cite{becker2013algorithm,bot2013algorithm}.

\subsection{Definitions, notation and some facts}\label{sec:notation}

In what follows, $\cH, \cG$, and $\vH$ denote (possibly infinite dimensional) Hilbert spaces. We always use the notations $\dotp{\cdot,\cdot}$ and $\|\cdot\|$ to denote the inner product and norm associated to a Hilbert space, respectively. Note that there is some ambiguity in this convention, but it simplifies the notation and no confusion should arise. The space $\vH$ will usually denote a product Hilbert space consisting of primal variables in $\cH$ and dual variables in $\cG$. Let $\vR_{++} = \{x \in \vR \mid x > 0\}$ denote the set of strictly positive real numbers. Let $\vN = \{k \in \vZ \mid k \geq 0\}$ denote the set of nonnegative integers. In all of the algorithms we consider, we utilize two stepsize sequences: the implicit sequence $(\gamma_j)_{j \in \vN} \subseteq \vR_{++}$ and the explicit sequence $(\lambda_j)_{j \in \vN} \subseteq \vR_{++}$.  We define the $k$-th partial sum of the sequence $(\gamma_j\lambda_j)_{j \in \vN}$ by the formula:
\begin{equation}\label{def:Lambda}\Sigma_k := \sum_{i=0}^k \gamma_i\lambda_i.
\end{equation}
Given a sequence  $(x^j)_{j \in \vN}\subset \cH$ and $k \in \vN$, we let $\overline{x}^k = ({1}/{\Sigma_k})\sum_{i=0}^k \gamma_i\lambda_i x^i$ denote its $k$th average  with respect to the sequence $(\gamma_j\lambda_j)_{j \in \vN}$.  We call a convergence result \emph{ergodic} if it is in terms of the sequence $(\overline{x}^j)_{j \in \vN}$, and \emph{nonergodic} if it is in terms of $(x^j)_{j \in \vN}$.

We denote the set of summable nonnegative sequences by $\ell^1_+(\vN) := \{(\eta_j)_{j \in \vN} \subseteq [0, \infty) \mid  \sum_{j = 0}^\infty \eta_j < \infty\}.$

The following definitions and facts are mostly standard and can be found in \cite{bauschke2011convex,combettes2013variable}

We let $\cB(\cH, \cG)$ denote the set of bounded linear maps from $\cH$ to $\cG$, and set $\cB(\cH) := \cB(\cH, \cH)$. We will use the notation $I_{\cH} \in \cB(\cH)$ to denote the identity map. Given a map $L \in \cB(\cH, \cG)$, we denote its adjoint by $L^\ast \in \cB(\cG, \cH)$. The operator norm on $L \in \cB(\cH, \cG)$ is defined by the following supremum: $\|L\| = \sup_{x\in \cH, \|x\|\leq 1} \|Lx\|$. Let $\rho \in \vR_+$ be a nonnegative real number. We let $\cS_\rho(\cH) \subseteq \cB(\cH)$ denote the set of linear $\rho$-strongly monotone self-adjoint maps: $$\cS_\rho(\cH) := \{ U \in \cB(\cH) \mid U = U^\ast, \left(\forall x \in \cH\right) \dotp{Ux, x} \geq \rho \|x\|^2\}.$$ We define the (semi)-norm and inner product induced by $U \in \cS_\rho(\cH)$ on $\cH$ by the formulae: for all $x, y \in \cH$, $\|x\|_U^2 := \dotp{Ux, x}$, and $\dotp{x, y}_U := \dotp{Ux, y}$.  The Loewner partial ordering on $\cS_\rho(\cH)$ is defined as follows: for all $U_1, U_2 \in \cS_\rho(\cH)$, we have 
\begin{align*}
U_1 \succcurlyeq U_2 && \Longleftrightarrow && \left(\forall x \in \cH\right) \; \|x\|_{U_1}^2 \geq \|x\|_{U_2}^2.  
\end{align*}

Let $L \geq 0$, and let $D$ be a nonempty subset of $\cH$.  A map $T : D \rightarrow \cH$ is called $L$-Lipschitz if for all $x, y \in D$, we have $\|Tx - Ty\| \leq L\|x-y\|$. In particular, $T$ is called \emph{nonexpansive} if it is $1$-Lipschitz. A map $N : D \rightarrow \cH$ is called $\lambda$-averaged \cite[Section 4.4]{bauschke2011convex} if there exists a nonexpansive map $T : D \rightarrow \cH$ and $\lambda \in (0, 1)$ such that 
\begin{align}\label{eq:averagednotation}
N = T_{\lambda}:= (1-\lambda) I_{\cH} + \lambda T.
\end{align}
A $(1/2)$-averaged map is called \emph{firmly nonexpansive}.  

Let $2^\cH$ denote the power set of $\cH$.  A set-valued operator $A : \cH \rightarrow 2^\cH$ is called \emph{monotone} if for all $x, y \in \cH$, $u \in Ax$, and $v \in Ay$, we have $\dotp{ x- y, u - v} \geq 0$. We denote the set of zeros of a monotone operator by $\zer(A) := \{x \in \cH \mid 0 \in Ax\}.$ The \emph{graph} of $A$ is denoted by $\gra(A) := \{(x, y) \mid x\in \cH, y \in Ax\}$. Evidently, $A$ is uniquely determined by its graph. A monotone operator $A$ is called \emph{maximal monotone} provided that $\gra(A)$ is not properly contained in the graph of any other monotone set-valued operator.  The \emph{inverse} of $A$, denoted by $A^{-1}$, is defined uniquely by its graph: $\gra(A^{-1}) := \{(y, x) \mid  x\in \cH, y \in Ax\}$. Let $\beta \in \vR_{++}$ be a positive real number. The operator $A$ is called \emph{$\beta$-strongly monotone} provided that for all $x, y \in \cH$,  $u \in Ax$, and $v \in Ay$, we have $\dotp{x - y, u - v} \geq \beta \|x-y\|^2$. A \emph{single-valued} operator $B : \cH \rightarrow 2^\cH$  maps each point in $\cH$ to a singleton and will be identified with the natural $\cH$-valued map it defines.  A single-valued operator $B$ is called \emph{$\beta$-cocoercive} provided that for all $x, y \in \cH$, we have $\dotp{x - y, Bx - By} \geq \beta \|Bx- By\|^2$.  Evidently, $B$ is $\beta$-cocoercive whenever $B^{-1}$ is $\beta$-strongly monotone.  The parallel sum of (not necessarily single-valued) monotone operators $A$ and $B$ is given by $A\square B := (A^{-1} + B^{-1})^{-1}$.  The \emph{resolvent} of a monotone operator $A$ is defined by the inversion $J_A := (I + A)^{-1}$.  Minty's theorem shows that  $J_A$ is single-valued and has full domain $\cH$ if, and only if,  $A$ is maximally monotone.  Note that $A$ is monotone if, and only if, $J_A$ is firmly nonexpansive.  Thus, the \emph{reflection operator}
\begin{align}\label{eq:refl}
\refl_{A} := 2J_A - I_{\cH}
\end{align}
 is nonexpansive on $\cH$ whenever $A$ is maximally monotone. If $\rho> 0$ and $U \in \cS_\rho(\cH)$, the operator $U^{-1}A$ is maximal monotone in $\dotp{\cdot, \cdot}_U$, if, and only if, $A$ is maximally monotone in $\dotp{\cdot, \cdot}$.  Let $\gamma \in (0, \infty)$. The resolvent of the map $\gamma U^{-1}A$ has the special identity: $J_{\gamma U^{-1}A} = U^{-1/2}J_{\gamma U^{-1/2} A U^{-1/2}} U^{1/2}$ \cite[Example 3.9]{combettes2012variable}.

Let $\Gamma_0(\cH)$ denote the set of closed, proper, and convex functions $f : \cH \rightarrow (-\infty, \infty]$.  Let $\dom(f) := \{x \in \cH \mid f(x) < \infty \}$. We will let $\partial f(x) : \cH \rightarrow 2^\cH$ denote the subdifferential of $f$: $\partial f(x) := \{ u\in \cH \mid \forall y \in \cH, f(y) \geq f(x) +  \dotp{y-x, u}\}$. We will always let 
\begin{align}\label{eq:thetildesub}
\tnabla f(x) \in \partial f(x) 
\end{align}
denote a subgradient of $f$ drawn at the point $x$, and the actual choice of the subgradient $\tnabla f(x)$ will always be clear from the context; note that this notation was also used in~\cite{bertsekasnotation}.  The subdifferential operator of $f$ is maximally monotone. The inverse of $\partial f$ is given by $\partial f^\ast$ where $f^\ast(y) := \sup_{x \in \cH} \{\dotp{y, x} - f(x)\}$ is the \emph{Fenchel conjugate} of $f$. If the function $f$ is $\beta$-strongly convex, then $\partial f$ is $\beta$-strongly monotone. 

 If a convex function $f : \cH \rightarrow (-\infty, \infty]$ is Fr{\'e}chet differentiable at $x \in \cH$, then $\partial f(x) = \{\nabla f(x)\}$. Suppose $f$ is convex and Fr{\'e}chet differentiable on $\cH$, and let $\beta \in \vR_{++}$ be a positive real number. Then the Baillon-Haddad theorem states that $\nabla f$ is $(1/\beta)$-Lipschitz, if, and only if, $\nabla f$ is $\beta$-cocoercive.
 
 The resolvent operator associated to $\partial f$ is called the \emph{proximal operator} and is uniquely defined by the following (strongly convex) minimization problem: $\prox_{f}(x) :=J_{\partial f}(x) = \argmin_{y \in \cH} \{f(y) + (1/2) \|y - x\|^2\}$. If $\rho> 0$, $U \in \cS_\rho(\cH)$, and $\gamma \in (0, \infty)$, the proximal operator of $f$ in the metric induced by $U$ is given by the following formula: for all $ x\in \cH$,
\begin{align}\label{eq:proximalmetric}
\prox_{\gamma f}^U(x) := J_{\gamma U^{-1} \partial f}(x) = \argmin_{y \in \cH} \left\{f(y) + \frac{1}{2\gamma} \|y - x\|_U^2\right\}.
\end{align}
The \emph{infimal convolution} of two functions $f, g: \cH \rightarrow (-\infty, \infty]$ is denoted by $f\square g: \cH \rightarrow [-\infty, \infty] : x \mapsto \inf_{y \in \cH} \{f(y) + g(x-y)\}$. The indicator function of a closed, convex set $C \subseteq \cH$ is denoted by $\iota_C : \cH \rightarrow 
\{0, \infty\}$; the indicator function is $0$ on $C$ and is $\infty$ on $\cH \backslash C$.

We will always use a $\ast$ superscript to denote a fixed point of a nonexpansive map, a zero of a monotone inclusion, or a minimizer of an optimization problem, e.g., $z^\ast$.

Finally, we call the following identity the \emph{cosine rule}:
\begin{align*}
\left(\forall x,y,z\in\cH\right) \qquad \|y-z\|^2+2\dotp{y-x,z-x}=\|y-x\|^2+\|z-x\|^2 \numberthis\label{eq:cosinerule}.
\end{align*}

\subsection{Assumptions}
~\\
\begin{assump}[Convexity]
Every function we consider is closed, proper, and convex. 
\end{assump}

Unless otherwise stated, a function is not necessarily differentiable. 

\begin{assump}[Differentiability]
Every differentiable function we consider is Fr{\'e}chet differentiable \cite[Definition 2.45]{bauschke2011convex}.
\end{assump}

We employ other assumptions throughout the paper, but we list them closer to where they are invoked.

\subsection{Basic properties of metrics}

A simple proof of the following Lemma recently appeared in \cite[Lemma 2.1]{combettes2013variable}.  It previously appeared in \cite[Section VI.2.6]{kato1995perturbation}.
\begin{lemma}[Metric properties]\label{lem:metricproperties}
Whenever $U, V \in \cS_{0}(\cH)$ satisfy the inequality $\alpha I_{\cH} \succcurlyeq U \succcurlyeq V \succcurlyeq \beta I_{\cH}$ for $\alpha, \beta > 0$, we have the ordering $(1/\beta)I_{\cH} \succcurlyeq V^{-1} \succcurlyeq U^{-1}\succcurlyeq (1/\alpha) I_{\cH}$, the inclusion $U^{-1} \in \cS_{\|U\|^{-1}}(\cH)$, and the inequality $\|U^{-1}\| \leq (1/\beta)$.
\end{lemma}

\subsection{Basic properties of resolvents and averaged operators}
~\\
The following are simple modifications of standard facts found in~\cite{bauschke2011convex}.
\begin{proposition}\label{prop:basicprox}
Let $\rho > 0$, let $\lambda > 0$, let $\alpha \in (0, 1)$, let $U \in \cS_\rho(\cH)$, let $A : \cH \rightarrow \cH$ be a single-valued maximal monotone operator, and let $f \in \Gamma_0(\cH)$
\begin{remunerate}
\item\label{prop:basicprox:part:J} \textbf{Optimality conditions of $J$:} We have $x^+ := J_{\gamma U^{-1}(\partial f + A)}(x)$ if, and only if, there exists a unique subgradient $\tnabla f(x^+) := (1/\gamma) U(x - x^+) - Ax^+ \in  \partial f(x^+)$,  such that
\begin{align*}
\tnabla f(x^+)  + Ax^+ = \frac{1}{\gamma} U(x- x^+) \in \partial f(x^+) + Ax^+.
\end{align*}
\item \label{prop:basicprox:part:contract} \textbf{Averaged operator contraction property:} Let $\lambda \in (0, 1)$. A map $T : \cH \rightarrow \cH$ is $\lambda$-averaged in the metric induced by $U$ if, and only if, for all $x, y \in \cH$,
\begin{align}\label{eq:avgdecrease}
\|Tx - Ty\|_U^2 \leq \|x - y\|_U^2 - \frac{1-\lambda}{\lambda} \|(I_{\cH}- T)x - (I_{\cH} - T)y\|_U^2.
\end{align}
\item \label{prop:basicprox:part:wider}\textbf{Wider relaxations:}  A map $T : \cH \rightarrow \cH$ is $\alpha$-averaged in $\|\cdot \|_U$, if, and only if, $T_{\lambda}$ (Equation~\eqref{eq:averagednotation}) is $\lambda \alpha$-averaged in $\|\cdot \|_U$ for all $\lambda \in (0, 1/\alpha)$.  In addition, $T_{1/\alpha}$ is nonexpansive with respect to $\|\cdot\|_U$.
\end{remunerate}
\end{proposition}

\subsection{Variable metrics}

Throughout this paper we will consider sequences of mappings  $(U_j)_{j \in \vN} \in \cS_\rho(\cH)$ for some $\rho > 0$.  In order to apply the standard convergence theory for variable metrics, we will make the following assumption:
\begin{assump}\label{assump:variablemetric}
There exists a summable sequence $(\eta_j)_{j \in \vN} \subseteq \ell_+^1(\vN)$ such that for all $k \in \vN$, $(1+\eta_k) U_k \succcurlyeq U_{k+1}.$ In addition $\mu := \sup_{j \in \vN} \|U_{j}\| < \infty$.
\end{assump}

Assumption~\ref{assump:variablemetric} is standard in variable metric algorithms \cite{combettes2013variable,doi:10.1080/01630563.2013.763825,combettes2012variable,parente2008class}.  

\begin{remark}
There is an asymmetry in our notation and the notation of \cite{combettes2013variable,doi:10.1080/01630563.2013.763825,combettes2012variable,parente2008class}. In our analysis, the map $U \in \cS_\rho(\cH)$ induces a metric on $\cH$. In other papers, the maps $U^{-1}$ induce a metric on $\cH$. 
\end{remark}

The following notation will be used throughout the rest of the paper. The proof is elementary.
\begin{proposition}[Metric parameters]
Suppose that $(\eta_j)_{j \in \vN} \subseteq \ell_+^1(\vN)$.  Define 
\begin{align*}
\eta_{\mathrm{p}} := \prod_{i=0}^\infty (1+\eta_i) && \mathrm{and} && \eta_{\mathrm{s}} := \sum_{i=0}^\infty \eta_i.
\end{align*}
Then $\eta_{\mathrm{p}}$ and $\eta_{\mathrm{s}}$ are finite.
\end{proposition}

The following Proposition is a consequence of the proof of \cite[Theorem 5.1]{combettes2013variable}. The proof is simple, so we omit it. 
\begin{proposition}\label{prop:variablemetricKM}
Let $\cH$ be a Hilbert space. Let $\rho \in (0, \infty)$, let $(\eta_j)_{j \in \vN} \subseteq \ell^1_+(\vN)$, and let $(U_j)_{j \in \vN} \in \cS_{\rho}(\cH)$ satisfy Assumption~\ref{assump:variablemetric}. For all $k \in \vN$, let $\alpha_k \in (0, 1)$, let $\lambda_k \in (0, 1/\alpha_k]$ be a relaxation parameter, and let $T_k : \cH \rightarrow \cH$ be $\alpha_k$-averaged in the metric $\|\cdot \|_{U_k}$. Furthermore, assume that there is a point $z^\ast \in \cH$ such that $T_kz^\ast = z^\ast$ for all $k \in \vN$. Let the $(z^j)_{j \in \vN}$ be generated by the following Krasnosel'ski\u{\i}-Mann (KM)-type iteration (Equation~\eqref{eq:averagednotation}): let $z^0 \in \cH$, and for all $k \in \vN$, define
\begin{align*}
z^{k+1} = (T_k)_{\lambda_k}z^k.
\end{align*}
Then the following are true:
\begin{remunerate}
\item \label{prop:variablemetricKM:part:fejer} For all $k \in \vN$, $\|z^{k+1} - z^\ast\|_{U_{k+1}}^2 \leq (1+\eta_k)\|z^k - z^\ast\|_{U_k}^2$ and hence,
\begin{align*}
\|z^k - z^\ast\|^2_{U_{k}} &\leq \eta_{\mathrm{p}}\|z^0 - z^\ast\|_{U_0}^2.
\end{align*}
\item \label{prop:variablemetricKM:part:FPR} The following sum is finite:
\begin{align*}
\sum_{i=0}^\infty \frac{1-\alpha_i\lambda_i}{\alpha_i\lambda_i} \|z^{i+1} - z^i\|^2 \leq \frac{1}{\rho}\left( 1+ \eta_{\mathrm{p}}\eta_{\mathrm{s}}\right)\|z^0 - z^\ast\|^2_{U_0}.
\end{align*}
\end{remunerate}
\end{proposition}

We will use the following proposition to select parameters in the FBS algorithm. The proof of the following fact follows from~\cite[Equation (3.35)]{vu2013splitting}

\begin{proposition}\label{prop:coco}
Let $\rho > 0$, let $\beta > 0$, let $B : \cH \rightarrow \cH$ be $\beta$-cocoercive in the norm $\|\cdot\|$, and let $U \in \cS_\rho(\cH)$.  Then $U^{-1}B$ is $\beta\rho$-cocoercive in the norm $\|\cdot\|_{U}$.
\end{proposition}

The following proposition essentially follows from the proof of \cite[Theorem 3.1]{doi:10.1080/01630563.2013.763825}. 
\begin{proposition}\label{prop:variablemetricFBF}
 Let $A : \cH \rightarrow 2^\cH$ be maximal monotone,  let $B : \cH \rightarrow \cH$ be monotone and $(1/\beta)$-Lipschitz for some $\beta > 0$, let $\rho > 0$, let $(U_j)_{j \in \vN} \subseteq \cS_{\rho}(\cH)$ satisfy Assumption~\ref{assump:variablemetric}, and let $(\gamma_j)_{j \in \vN} \subseteq (0, \rho\beta]$.  Let $(z^j)_{j \in \vN}$ be a sequence of points defined by the iteration:  let $z^0 \in \cH$ and for all $k \in \vN$, define
 \begin{align*}
y^k &= z^k - \gamma_k U_k^{-1} Bz^k;\\
x^k &= J_{\gamma_k U_k^{-1} A}(y^k); \\
w^k &= x^k - \gamma_k U_k^{-1} Bx^k; \\
z^{k+1} &= z^k - y^{k} + w^k.
 \end{align*}
 Suppose that $\zer(A + B) \neq \emptyset$. Then for all $z^\ast \in \zer(A + B)$ and for all $k \in \vN$, we have, $\|z^{k+1} - z^\ast\|_{U_{k+1}}^2 \leq (1+\eta_k)\|z^k - z^\ast\|_{U_k}^2$.
\end{proposition}


\section{The unifying scheme}\label{sec:US}

In this section, we introduce a prototype monotone inclusion problem that generalizes and summarizes many primal-dual problem formulations found in the literature.  After we describe the problem, we will introduce an abstract unifying scheme that generalizes many existing primal-dual algorithms.  We will describe how to measure convergence of the unifying scheme, and introduce a fundamental inequality that bounds our measure of convergence. Finally, we will identify the key terms in the fundamental inequality and simplify them in the case of several abstract splitting algorithms. 

In Section~\ref{sec:applications}, we will show that this unifying scheme relates to many existing algorithms, and extend the convergence rate results of those methods.

\subsection{Problem and algorithm}

We focus on the following problem:

\begin{problem}[Prototype primal-dual problem]\label{prob:PDprototype}
Let $(\vH, \dotp{\cdot, \cdot})$ be a Hilbert space, let $\vf, \vg \in \Gamma_0(\vH)$, and let $\vS : \vH \rightarrow \vH$ be a skew symmetric map: $\vS^\ast = -\vS$. Then the prototype primal-dual problem is to find $\vx^\ast \in \vH$ such that
\begin{align}\label{eq:PDinclusion}
0 \in \partial \vf(\vx^\ast) + \partial \vg(\vx^\ast) + \vS \vx^\ast.
\end{align}
\end{problem}

Evidently, Problem~\ref{prob:PDprototype} is a \emph{monotone inclusion problem} because $\partial \vf, \partial \vg$, and $\vS$ are maximally monotone operators on $\vH$ \cite[Example 20.30]{bauschke2011convex}. 

We are now ready to define our unifying scheme.  

\begin{algorithm}[H]
\SetKwInOut{Input}{input}\SetKwInOut{Output}{output}
\SetKwComment{Comment}{}{}
\BlankLine
\Input{$\vz^0 \in \vH; (\lambda_j)_{j\geq 0} \subseteq \vR_{++}; (\gamma_j)_{j \in \vN} \subseteq \vR_{++};  \rho > 0; (U_j)_{j \in \vN} \subseteq \cS_\rho(\vH).$}
\For{$k=0,~1,\ldots$}{
$\vz^{k+1} = \vz^k - \gamma_k\lambda_k U_k^{-1} \left( \tnabla \vf(\vx_\vf^k) + \tnabla \vg(\vx_\vg^k) + \vS\vx_{\vS}^k\right)$\;}
\caption{{Unifying scheme}}
\label{alg:US}
\end{algorithm}

Note that the points $\vx_\vf^k, \vx_\vg^k$, and $\vx_\vS^k$ as well as the subgradients $\tnabla \vf(\vx_\vf^k) \in \partial \vf(\vx_\vf^k)$ and $\tnabla \vg(\vx_\vg^k) \in \partial \vg(\vx_\vg^k)$ are unspecified in the description of Algorithm~\ref{alg:US}.  In the algorithms we study, these points and subgradients will be generated by proximal and forward gradient operators and, thus, can be determined given $\vz^k$; see Section~\ref{sec:examplesofUS} for examples. However, Algorithm~\ref{alg:US} is only meant to illustrate the algebraic form that our analysis addresses, and it is not meant to be an actual algorithm that solves Problem~\ref{eq:PDinclusion}. The positive scalar sequence $(\lambda_j)_{j \in \vN}$ consists of \emph{relaxation parameters}, or explicit stepsize parameters, whereas the sequence $(\gamma_j)_{j \in \vN}$ consists of \emph{proximal parameters}, or implicit stepsize parameters. The strongly monotone maps $(U_j)_{j \in \vN}$ induce the metrics used in each iteration of the algorithm.

In all of our applications, $\vH$ will be a product space of primal and dual variables. In this setting, $\vf$ and $\vg$ will be block-separable maps, and $\vg$ will sometimes be differentiable. The map $\vS$ ``mixes" the primal and dual variable sequences in the product space.  Mixing is necessary, because the sequences are otherwise uncoupled.  

The sequence of maps $(U_j)_{j \in \vN}$ is employed for two purposes. First, the maps are used because the evaluation of the resolvent $J_{\partial \vf + \vS}$, which is a basic building block of most of the algorithms we study, may not be simple. Thus, the primal-dual algorithms that we study formulate special metrics induced by $U \in \cS_\rho(\vH)$ such that $J_{U^{-1}(\partial \vf + \vS)}$ is as easy to evaluate as $\prox_{\vf}$ (See Section~\ref{sec:applications}). Hence, in our analysis we must at least consider fixed metrics that are different from the standard product metric on $\vH$. Second, we allow the metrics to vary at each iteration because it can significantly improve the practical performance of the algorithm, e.g., by employing second order information, or even simple time-varying diagonal metrics \cite{pock2011diagonal,goldstein2013adaptive}.

\subsection{Examples of the unifying scheme}\label{sec:examplesofUS}
In this section we introduce four algorithms and show that they are special cases of Algorithm~\ref{alg:US}. We will also introduce several assumptions on the algorithm parameters that ensure convergence.  These assumptions will remain in effect throughout the rest of the paper.  Note that the convergence theory of the methods in this section is well-studied. See~\cite{bauschke2011convex,combettes2013variable,vu2013splitting,combettes2012variable,rockafellar1976monotone,tseng2000modified,lions1979splitting} for background. Finally, we will say that several algorithms in this section are \emph{relaxed}. For brevity, we will drop this adjective whenever convenient.

The relaxed variable metric PPA applies to problems in which  $\vg \equiv 0$.

\begin{algorithm}[H]
\SetKwInOut{Input}{input}\SetKwInOut{Output}{output}
\SetKwComment{Comment}{}{}
\BlankLine
\Input{$\vz^0 \in \vH; (\lambda_j)_{j\geq 0} \subseteq (0,2]; (\gamma_j)_{j \in \vN} \subseteq \vR_{++};  \rho > 0; (U_j)_{j \in \vN} \subseteq \cS_\rho(\vH).$}
\For{$k=0,~1,\ldots$}{
$\vx_\vg^k = \vz^k$\;
$\vx_\vf^k  = J_{\gamma_kU_k^{-1}(\partial \vf + \vS)}(\vz^k)$\;
$\vz^{k+1} = (1-\lambda_k) \vz^k + \lambda_k\vx_\vf^k$\;}
\caption{{Relaxed variable metric proximal point algorithm (PPA)}}
\label{alg:PPA}
\end{algorithm}

The relaxed variable metric FBS algorithm can be applied whenever $\vg$ is differentiable and $\nabla \vg$ is $({1}/{\beta})$-Lipschitz for some $\beta > 0$. 

\begin{algorithm}[H]
\SetKwInOut{Input}{input}\SetKwInOut{Output}{output}
\SetKwComment{Comment}{}{}
\BlankLine
\Input{$\vz^0 \in \vH; \rho > 0; \varepsilon \in (0, 2\beta\rho);$ \\ $(\gamma_j)_{j \in \vN} \subseteq (0, 2\beta\rho - \varepsilon];$  \\$\alpha_{k} := (2\beta\rho)/(4\beta\rho - \gamma_k)$ for $k \in \vN$; \\ $ \delta \in (0, \inf\{1/\alpha_j \mid j \in \vN\});\quad$ \textit{//comment: this interval is nonempty} \\ $  \lambda_k \in (0,1/\alpha_{k} - \delta]$ for $k \in \vN;$ \\$ (U_j)_{j \in \vN} \subseteq \cS_\rho(\vH).$}
\For{$k=0,~1,\ldots$}{
$\vx_\vg^k = \vz^k$\;
$\vx_\vf^k  = J_{\gamma_kU_k^{-1}(\partial \vf + \vS)}(\vz^k - \gamma_kU_k^{-1} \nabla \vg(\vz^k))$\;
$\vz^{k+1} = (1-\lambda_k) \vz^k + \lambda_k\vx_\vf^k$\;}
\caption{{Relaxed variable metric forward-backward algorithm (FBS)}}
\label{alg:PDFBS}
\end{algorithm}

In the relaxed PRS algorithm, we fix the metric and the implicit stepsize parameters throughout the course of the algorithm.  We do this because the fixed-points of the PRS operator can vary with $\gamma$ and $U$.  Thus, changing these parameters will lead to an algorithm that ``chases" a new fixed-point at each iteration. 

\begin{algorithm}[H]
\SetKwInOut{Input}{input}\SetKwInOut{Output}{output}
\SetKwComment{Comment}{}{}
\BlankLine
\Input{$\vz^0 \in \vH; (\lambda_j)_{j \in \vN} \subseteq (0, 2]; \gamma > 0; \rho > 0; U \in \cS_\rho(\vH); w \in \vR.$}
\For{$k=0,~1,\ldots$}{
$\vz^{k+1} = (1-\frac{\lambda_k}{2})\vz^k + \frac{\lambda_k}{2}\refl_{\gamma U^{-1}(\partial \vf + w\vS)} \circ \refl_{\gamma U^{-1}(\partial \vg + (1-w)\vS)}(\vz^k)$\;}
\caption{{Relaxed Peaceman-Rachford splitting (PRS)}}
\label{alg:PDPRS}
\end{algorithm}

The variable metric FBF algorithm can be applied whenever $\vg$ is differentiable and $\nabla \vg$ is $({1}/{\beta})$-Lipschitz for some $\beta > 0$.

\begin{algorithm}[H]
\SetKwInOut{Input}{input}\SetKwInOut{Output}{output}
\SetKwComment{Comment}{}{}
\BlankLine
\Input{$\vz^0 \in \vH;  \rho > 0; (U_j)_{j \in \vN} \subseteq \cS_\rho(\vH);  (\gamma_j)_{j \in \vN} \subseteq (0, \rho/\left(\beta^{-1} + \|\vS\|\right)).$}
\For{$k=0,~1,\ldots$}{
$\vy^k = \vz^k - \gamma_kU_k^{-1}(\nabla \vg(\vz^k) + \vS\vz^k)$\;
$\vx_\vf^k = J_{\gamma_kU_k^{-1}\partial \vf}(\vy^k)$\;
$\vw^k = \vx_\vf^k - \gamma_kU_k^{-1}(\nabla \vg(\vx_\vf^k) + \vS\vx_\vf^k)$\;
$\vz^{k+1} = \vz^k - \vy^k + \vw^k;$}
\caption{{Variable metric forward-backward-forward algorithm (FBF)}}
\label{alg:PDFBF}
\end{algorithm}

The following lemma relates the above algorithms to the unifying scheme.
\begin{lemma}\label{lem:spunif}
Algorithms~\ref{alg:PPA},~\ref{alg:PDFBS},~\ref{alg:PDPRS}, and~\ref{alg:PDFBF} are special cases of the unifying scheme. In particular, the following hold for all $k \in \vN$:
\begin{remunerate}
\item\label{lem:spunif:part:PPA} In Algorithm~\ref{alg:PPA}, we have $\vx_\vg^k := \vz^k$, $\vx_\vS^k := \vx_\vf^k,$ and $$\tnabla \vf(\vx_\vf^k) := (1/\gamma_k)U_k(\vz^k - \vx_\vf^k) - \vS\vx_\vf^k \in \partial \vf(\vx_\vf^k).$$ 
\item\label{lem:spunif:part:PDFBS} In Algorithm~\ref{alg:PDFBS}, we have $\vx_\vg^k := \vz^k$, $\vx_\vS^k := \vx_\vf^k$, and $$\tnabla f(\vx_\vf^k) := (1/\gamma_k)U_k(\vz^k - \gamma_k U^{-1}_k \nabla g(\vz^k) - \vx_\vf^k) - \vS\vx_\vf^k \in \partial \vf(\vx_\vf^k).$$
\item\label{lem:spunif:part:PDPRS} In Algorithm~\ref{alg:PDPRS}, we have $\vz^{k+1} - \vz^{k} = \lambda_k(\vx_{\vf}^k - \vx_{\vg}^k)$ for
\begin{align*}
\vx_\vg^k := J_{\gamma U^{-1}(\partial \vg + (1-w)\vS)}(\vz^k); && && \vx_\vf^k := J_{\gamma U^{-1}(\partial \vf + w\vS)} \circ\refl_{\gamma U^{-1}(\partial \vg + (1-w)\vS)}(\vz^k);\\
 \vx_\vS^k := w\vx_\vf^k + (1-w) \vx_\vg^k;  && &&  \tnabla \vg(\vx_\vg^k): = (1/\gamma)U(\vz^k - \vx_\vg^k)  - (1-w)\vS\vx_\vg^k \in \partial \vg(\vx_\vg^k); 
\end{align*}
and $\tnabla \vf(\vx_\vf^k) := (1/\gamma)U(2\vx_\vg^k - \vz^k - \vx_\vf^k) - w\vS\vx_\vf^k\in \partial \vf(\vx_\vf^k)$. 
\item\label{lem:spunif:part:PDFBF} In Algorithm~\ref{alg:PDFBF}, we have $\lambda_k = 1$, $\vx_\vg^k := \vx_\vf^k$, $\vx_\vS^k := \vx_\vf^k$, and $$\tnabla \vf(\vx_\vf^k) := (1/\gamma_k) U_k(\vy^k - \vx_\vf^k) \in \partial \vf(\vx_\vf^k) .$$ 
\end{remunerate}
\end{lemma}
{\em Proof.}
Fix $k \in \vN$, and note that the subgradient identities all follow from Part~\ref{prop:basicprox:part:J} of Proposition~\ref{prop:basicprox}.

Part~\ref{lem:spunif:part:PPA}: This is immediate.

Part~\ref{lem:spunif:part:PDFBS}: From Part~\ref{prop:basicprox:part:J} of Proposition~\ref{prop:basicprox},  we have the following identity: 
\begin{align*}
\vx_\vf^k  &= \vz^k - \gamma_kU_k^{-1} \left( \tnabla \vf(\vx_\vf^k) + \nabla \vg(\vx_\vg^k) + \vS\vx_{\vS}^k\right).
\end{align*}
Thus, altogether we have $\vz^{k+1} = \vz^k - \gamma_k\lambda_k U_k^{-1} \left( \tnabla \vf(\vx_\vf^k) + \nabla \vg(\vx_\vg^k) + \vS\vx_{\vS}^k\right).$

Part~\ref{lem:spunif:part:PDPRS}: We have
\begin{align*}
&\refl_{\gamma U^{-1}(\partial \vf + w\vS)} \circ \refl_{\gamma U^{-1}(\partial \vg + (1-w)\vS)}(\vz^k) \\
&= \refl_{\gamma U^{-1}(\partial \vf + w\vS)} ( \vz^k - 2\gamma U^{-1}(\tnabla  \vg(\vx_\vg^k) + (1-w)\vS\vx_\vg^k))\\
&=  \vz^k - 2\gamma U^{-1}(\tnabla \vf(\vx_\vf^k) + \tnabla  \vg(\vx_\vg^k) + \vS(w\vx_\vf^k + (1-w)\vx_\vg^k)).
\end{align*}
Therefore, if we define $\vx_\vS^k := w\vx_\vf^k + (1-w)\vx_\vg^k$, then
\begin{align*}
\vz^{k+1} &= \vz^k - \gamma \lambda_k U^{-1} \left( \tnabla \vf(\vx_\vf^k) + \tnabla \vg(\vx_\vg^k) + \vS\vx_{\vS}^k\right).
\end{align*}

Part~\ref{lem:spunif:part:PDFBF}: We have
\begin{align*}
\vz^{k+1} - \vz^k = \vw^k - \vy^k = \vw^k - \vx_\vf^k + \vx_\vf^k - \vy^k = -\gamma_kU_k^{-1} \left( \tnabla \vf(\vx_\vf^k) + \nabla \vg(\vx_\vf^k) + \vS\vx_{\vf}^k\right). \endproof
\end{align*}

\subsubsection{Convergence properties}

Now we establish two basic and well known results on the boundedness and summability of various terms related to the above algorithms. These facts will be used repeatedly in our convergence rate analysis.

\begin{proposition}[Averagedness properties]\label{prop:avgeraged}
Let $\rho \in \vR_{++}$, let $U \in \cS_\rho(\vH)$, let $\gamma \in \vR_{++}$, and let $\beta \in \vR_{++}$. Then the following hold:
\begin{remunerate}
\item\label{prop:avgeraged:part:resolve} The operator $J_{\gamma U ^{-1}(\partial \vf + \vS)}$ is $(1/2)$-averaged in the norm $\|\cdot\|_{U}$. In addition, the set of fixed points of $J_{\gamma U ^{-1}(\partial \vf + \vS)}$ is equal to $\zer\left(\partial \vf + \vS \right)$
\item\label{prop:avgeraged:part:FBS} Let $\gamma \in (0, 2\beta\rho)$. Suppose that $\vg$ is differentiable and $\nabla \vg$ is $(1/\beta)$-Lipschitz.  Then the composition
\begin{align}\label{eq:forwardbackwardoperator}
T^{U, \gamma}_{\mathrm{FBS}} :=  J_{\gamma U ^{-1}(\partial \vf + \vS)} \circ (I_{\vH}- \gamma U^{-1} \nabla \vg)
\end{align}
is $\alpha_{\rho, \gamma}$-averaged in the norm $\|\cdot\|_{U}$ where 
\begin{align}\label{eq:alphafbs}
\alpha_{\rho, \gamma} &:= \frac{2\beta\rho}{4\beta\rho - \gamma}.
\end{align}
In addition, the set of fixed points of $T^{U, \gamma}_{\mathrm{FBS}}$ is equal to $\zer\left(\partial \vf + \nabla \vg + \vS \right)$.
\item\label{prop:avgeraged:part:PRS} Let $w \in \vR$, and define the PRS operator:
\begin{align}\label{eq:TPRS}
\TPRS := \refl_{\gamma U^{-1}(\partial \vf + w\vS)} \circ \refl_{\gamma U^{-1}(\partial \vg + (1-w)\vS)}.
\end{align}
Then $\TPRS$ is nonexpansive in the metric $\|\cdot \|_U$.  Thus, the following DRS operator
\begin{align*}
(\TPRS)_{1/2} &= \frac{1}{2}I_{\vH} + \frac{1}{2} \refl_{\gamma U^{-1}(\partial \vf + w\vS)} \circ \refl_{\gamma U^{-1}(\partial \vg + (1-w)\vS)} \numberthis\label{eq:DRSoperator}
\end{align*}
is $(1/2)$-averaged. In addition, the set of fixed points of $\TPRS$ and $(\TPRS)_{1/2}$ coincide and $\zer(\partial \vf + \partial \vg + \vS) = \{J_{\gamma U^{-1}(\partial \vg + (1-w)\vS)}(\vz) \mid \vz \in \vH \; \mathrm{and} \;  \TPRS\vz = \vz\}$.
\end{remunerate}
\end{proposition}
\begin{proof}
Parts~\ref{prop:avgeraged:part:resolve} and~\ref{prop:avgeraged:part:PRS} are simple modifications of standard facts found in~\cite{bauschke2011convex}.

Part~\ref{prop:avgeraged:part:FBS}: Note that $U^{-1} \nabla \vg$ is $\beta\rho$-cocoercive in $\|\cdot\|_U$ by Proposition~\ref{prop:coco} and the Baillon-Haddad theorem \cite{baillon1977quelques}.  Thus, $I_{\vH} - \gamma U^{-1} \nabla \vg$ is $\gamma/(2\beta\rho)$ averaged in $\|\cdot\|_U$ by~\cite[Proposition 4.33]{bauschke2011convex}. Thus, the formula for $\alpha_{\rho, \gamma}$ follows from \cite[Theorem 3(b)]{nobhuiko}.  The fixed-point identity follows from a simple modification of~\cite[Theorem 25.1]{bauschke2011convex}.
\end{proof}~\\

\begin{proposition}[Bounded and summable sequences]\label{prop:bddsum}
The following hold:
\begin{remunerate}
\item\label{prop:bddsum:part:PPA} Let $\vz^\ast \in \zer(\partial \vf +  \vS)$. Then in Algorithm~\ref{alg:PPA}, we have for all $k \in \vN$, that $\|\vz^{k+1} - \vz^\ast\|_{U_{k+1}}^2 \leq (1+\eta_k)\|\vz^k - \vz^\ast\|_{U_k}^2$ and hence, $\|\vz^k - \vz^\ast\|^2_{U_{k}} \leq \eta_{\mathrm{p}}\|\vz^0 - \vz^\ast\|_{U_0}^2.$
\item\label{prop:bddsum:part:PDFBS}  Let $\vz^\ast \in \zer(\partial \vf + \nabla \vg + \vS)$. Then in Algorithm~\ref{alg:PDFBS}, the following are true:
\begin{romannum}
\item \label{prop:bddsum:part:PDFBS:bdd}For all $k \in \vN$, 
$\|\vz^{k+1} - \vz^\ast\|_{U_{k+1}}^2 \leq (1+\eta_k)\|\vz^k - \vz^\ast\|_{U_k}^2$ and hence, $\|\vz^{k} - \vz^\ast\|_{U_{k}}^2 \leq \eta_{\mathrm{p}}\|\vz^0 - \vz^\ast\|_{U_0}^2.$
\item\label{prop:bddsum:part:PDFBS:sum}
The following sum is finite:
\begin{align}\label{eq:FBSsumbound}
\sum_{i=0}^\infty \frac{1-\alpha_i\lambda_i}{\alpha_i\lambda_i} \|\vz^{i+1} - \vz^i\|^2 \leq \frac{1}{\rho}\left( 1+ \eta_{\mathrm{p}}\eta_{\mathrm{s}}\right)\|\vz^0 - \vz^\ast\|^2_{U_0}.
\end{align}
\end{romannum}
\item\label{prop:bddsum:part:PDPRS} Let $\vz^\ast$ be a fixed-point of $\TPRS$. Then in Algorithm~\ref{alg:PDPRS}, we have for all $k \in \vN$, that $\|\vz^{k+1} - \vz^\ast\|_{U}^2 \leq \|\vz^k - \vz^\ast\|_{U}^2$ and hence, $\|\vz^{k} - \vz^\ast\|_{U}^2 \leq \|\vz^0 - \vz^\ast\|_U^2$.
\item \label{prop:bddsum:part:PDFBF} Let $\vz^\ast \in \zer(\partial \vf + \nabla \vg + \vS)$. Then in Algorithm~\ref{alg:PDFBF}, we have for all $k \in \vN$ that $\|\vz^{k+1} - \vz^\ast\|_{U_{k+1}}^2 \leq (1+\eta_k)\|\vz^k - \vz^\ast\|_{U_k}^2$ and hence, $\|\vz^k - \vz^\ast\|^2_{U_{k}} \leq \eta_{\mathrm{p}}\|\vz^0 - \vz^\ast\|_{U_0}^2.$
\end{remunerate}
\end{proposition}
\begin{proof}
Parts~\ref{prop:bddsum:part:PPA},~\ref{prop:bddsum:part:PDFBS}, and~\ref{prop:bddsum:part:PDPRS} follow from Proposition~\ref{prop:variablemetricKM} applied to the sequences of operators $(T_j)_{j \in \vN} := (J_{\gamma_j U_j^{-1}(\partial \vf + \vS)})_{j \in \vN}$, $(T_j)_{j \in \vN} := (T_{\mathrm{FBS}}^{U_j, \gamma_j})_{j \in \vN}$, and $(T_j)_{j \in \vN} := ((\TPRS)_{1/2})_{j \in \vN}$, respectively.

Part~\ref{prop:bddsum:part:PDFBF} follows from Proposition~\ref{prop:variablemetricFBF} applied to the the maximal monotone operator $\partial \vf$ and the $(\beta^{-1} + \|\vS\|)$-Lipschitz operator $\nabla \vg + \vS$.
\end{proof}

\subsection{The fundamental inequality}

This section describes the pre-primal-dual gap (Definition~\ref{defi:preprimaldualgap}).  We use the pre-primal-dual gap to measure the convergence of the unifying scheme.  In Section~\ref{sec:applications}, we will show that under certain conditions, the pre-primal-dual gap function bounds the primal and dual objective errors of the iterates generated by a class of primal-dual algorithms.

Before we introduce the gap function, we analyze the optimality conditions of Problem~\ref{prob:PDprototype}. The following lemma is well-known.
\begin{lemma}\label{lem:variationalinequality}
Let $\vx^\ast \in \vH$. Suppose that $\vx^\ast$ solves Problem~\ref{prob:PDprototype}. Then for all $\vx \in \vH,$ 
\begin{align}\label{eq:variationalinequality}
 \vf(\vx) + \vg(\vx) + \dotp{\vS\vx, -\vx^\ast} - \vf(\vx^\ast) - \vg(\vx^\ast) &\geq 0.
\end{align}

On the other hand, if $\partial (\vf+ \vg)(\vx^\ast) = \partial \vf (\vx^\ast)+ \partial \vg(\vx^\ast)$ and $\vx^\ast$ satisfies Equation~\eqref{eq:variationalinequality} for all $\vx \in \dom(\vf)\cap\dom(\vg)$, then $\vx^\ast$ solves Problem~\ref{prob:PDprototype}.
\end{lemma}
\begin{proof}
If $\vx^\ast$ solves Problem~\ref{prob:PDprototype}, then $-\vS\vx^\ast$ is a subgradient of $\vf + \vg$ at the point $\vx^\ast$. Thus, Equation~\eqref{eq:variationalinequality} follows after noting that $\dotp{\vS\vx, \vx} = 0$ for all $\vx \in \vH$.

The other direction follows because Equation~\eqref{eq:variationalinequality} characterizes the set of subgradients of the form $-\vS\vx^\ast \in \partial (\vf + \vg)(\vx^\ast) = \partial \vf(\vx^\ast) + \partial \vg(\vx^\ast)$.
\end{proof}

See \cite[Corollary 16.38]{bauschke2011convex} for conditions that imply additivity of the subdifferential.

Lemma~\ref{lem:variationalinequality} motivates the following definition:

\begin{definition}[Pre-primal-dual gap]\label{defi:preprimaldualgap}
Let the setting be as in Algorithm~\ref{alg:US}.  Define the pre-primal dual gap function by the formula: for all $\vx_\vf, \vx_\vg, \vx_{\vS}, \vx \in \vH$, let
\begin{align}
\cG^{\mathrm{pre}}(\vx_\vf, \vx_\vg, \vx_{\vS}; \vx) &=  \vf(\vx_\vf) + \vg(\vx_\vg) + \dotp{\vS\vx_{\vS}, -\vx} -  \vf(\vx) -  \vg(\vx).
\end{align}
\end{definition}
\indent We name $\cG^{\mathrm{pre}}$ the pre-primal-dual-gap function after the standard primal-dual gap function that appears in \cite{chambolle2011first,bo2014convergence,boct2014convergence}.  We use the word ``pre" because the standard primal-dual gap function usually involves a supremum over the last variable $\vx$.  Note that if $\partial (\vf + \vg)(\vx') = \partial \vf(\vx') + \partial \vg(\vx')$ and
\begin{align}\label{eq:ppdgsup}
\sup_{\vx \in \vH} \cG^{\mathrm{pre}}(\vx', \vx', \vx'; \vx) \leq 0,
\end{align}
then $\vx'$ is a solution of Problem~\ref{prob:PDprototype} (Lemma~\ref{lem:variationalinequality}).

Our goal throughout the rest of this paper is to bound the pre-primal-dual gap when $\vx_\vf = \vx_\vg = \vx_\vS$.  Because of Equation~\eqref{eq:ppdgsup}, all of our upper bounds will be a function of the norm of the last component of $\cG^{\mathrm{pre}}$. In some cases, we can restrict the supremum in Equation~\eqref{eq:ppdgsup} to a smaller subset $C \subseteq \vH$. This is the case if, for example, $\dom(\vf) \cap \dom(\vg)$ is bounded. Whenever the supremum can be restricted, we obtain a meaningful convergence rate.

Finally, Lemma~\ref{lem:variationalinequality} shows that for all $\vx \in \vH$,
\begin{align}\label{eq:positivegapfunction}
\cG^{\mathrm{pre}}(\vx, \vx, \vx; \vx^\ast) \geq 0
\end{align}
whenever $\vx^\ast$ solves Problem~\ref{prob:PDprototype}. See Section~\ref{sec:preprimaldualprimal} for other lower bounds of the pre-primal-dual gap in the context of a particular convex optimization problem.

The following is our main tool to bound the pre-primal-dual gap.  
\begin{proposition}[Upper fundamental inequality for primal dual schemes]\label{prop:PDupper}
Suppose that $(\vz^j)_{j\geq 0}$ is generated by Algorithm~\ref{alg:US}, and let $\vx\in \vH$. Then the following inequality holds: for all $k \in \vN$,
\begin{align*}
2\gamma_k\lambda_k \cG^{\mathrm{pre}}(\vx_\vf^k, \vx_\vg^k, \vx_{\vS}^k; \vx)  &\leq \|\vz^{k} - \vx\|_{U_k}^2 - \|\vz^{k+1} - \vx\|_{U_k}^2 - \|\vz^{k+1} - \vz^k\|_{U_k}^2 \\
&+ 2\gamma_k\lambda_k \dotp{\vx_{\vf}^k - \vz^{k+1}, \tnabla \vf(\vx_{\vf}^k)} \\
&+ 2\gamma_k\lambda_k  \dotp{\vx_{\vg}^k - \vz^{k+1}, \tnabla \vg(\vx_{\vg}^k)}\\
&+ 2\gamma_k\lambda_k  \dotp{-\vz^{k+1}, \vS\vx_{\vS}^k}. \numberthis \label{eq:PDupper}
\end{align*} 
\end{proposition}
\begin{proof}
Fix $k \in \vN$. First expand the norm:
\begin{align*}
\|\vz^{k+1} - \vx\|_{U_k}^2 &= \|\vz^{k} - \vx\|_{U_k}^2 + 2\dotp{\vx - \vz^{k+1}, \vz^{k} - \vz^{k+1}}_{U_k} - \|\vz^{k+1} - \vz^k\|^2_{U_k}.
\end{align*}
Now we expand the inner product:
\begin{align*}
2\dotp{\vx - \vz^{k+1}, \vz^{k} - \vz^{k+1}}_{U_k} &= 2\dotp{\vx - \vz^{k+1},  \gamma_k\lambda_k U_k^{-1} \left( \tnabla \vf(\vx_\vf^k) + \tnabla \vg(\vx_\vg^k) + \vS\vx_{\vS}^k\right)}_{U_k} \\
&= 2\gamma_k\lambda_k\dotp{\vx - \vz^{k+1}, \tnabla \vf(\vx_\vf^k)} + 2\gamma_k\lambda_k\dotp{\vx - \vz^{k+1}, \tnabla \vg(\vx_\vg^k)} \\
&+ 2\gamma_k\lambda_k\dotp{\vx - \vz^{k+1}, \vS\vx_\vS^k}.
\end{align*}
We add and subtract a point in the inner products involving $\vf$ and $\vg$ and use the subgradient inequality to get:
\begin{align*}
2\gamma_k\lambda_k\dotp{\vx - \vz^{k+1}, \tnabla \vf(\vx_\vf^k)} &\leq 2\gamma_k\lambda_k \dotp{\vx_{\vf}^k - \vz^{k+1}, \tnabla \vf(\vx_{\vf}^k)} + 2\gamma_k\lambda_k(\vf(\vx) - \vf(\vx_\vf^k)); \\
2\gamma_k\lambda_k\dotp{\vx - \vz^{k+1}, \tnabla \vg(\vx_\vg^k)} &\leq 2\gamma_k\lambda_k \dotp{\vx_{\vg}^k - \vz^{k+1}, \tnabla \vg(\vx_{\vg}^k)} + 2\gamma_k\lambda_k(\vg(\vx) - \vg(\vx_\vg^k)).
\end{align*}
Therefore Equation~\eqref{eq:PDupper} follows after rearranging.
\end{proof}

The upper fundamental inequality in Proposition~\ref{prop:PDupper} bounds the pre-primal-dual gap with the sum of an alternating sequence and a key term.

\begin{definition}[Upper key term]\label{defi:PDkeyterm}
Let $(\vz^j)_{j \in \vN}$ be generated by Algorithm~\ref{alg:US}.  For all $k \in \vN$, we define the fundamental upper key term
\begin{align*}
\kappa_u^k(\lambda_k) := &- \|\vz^{k+1} - \vz^k\|_{U_k}^2 \\
&+ 2\gamma_k\lambda_k \dotp{\vx_{\vf}^k - \vz^{k+1}, \tnabla \vf(\vx_{\vf}^k)}  \\
&+ 2\gamma_k\lambda_k  \dotp{\vx_{\vg}^k - \vz^{k+1}, \tnabla \vg(\vx_{\vg}^k)} \\
&+ 2\gamma_k\lambda_k   \dotp{-\vz^{k+1}, \vS\vx_{\vS}^k}. \numberthis \label{eq:PDupperkeyterm}
\end{align*}
\end{definition}

The value $\kappa_u^k(\lambda_k)$ depends on the entire history of Algorithm~\ref{alg:US} up to and including iteration $k$, but in our analysis we will only view $\kappa_u^k(\lambda_k)$ as a function of the parameter $\lambda_k$. Throughout the rest of the paper, we will often make the dependence of the upper key term on $\lambda_k$ implicit, and denote $\kappa_u^k := \kappa_u^k(\lambda_k)$. However, in the proof of Theorem~\ref{sec:nonergodic} we will need to keep the dependence explicit. 

\subsubsection{Computing the  upper key terms}

The following proposition will compute the upper key terms induced by the PPA, FBS,  PRS, and FBF algorithms. See Section~\ref{sec:examplesofUS} for the definitions of the points $\vx_\vf^k, \vx_\vg^k$, and $\vx_\vS^k$.

\begin{proposition}[Computing the upper key terms]\label{prop:upperkeyterms}
Let $(\vz^j)_{j \in \vN}$ be generated by Algorithm~\ref{alg:US}.
Then for all $k \in \vN$, the following inequalities and identities hold: 
\begin{remunerate}
\item\label{prop:upperkeyterms:part:PPA} In Algorithm~\ref{alg:PPA}, we have $\kappa_u^k(\lambda_k) =  \left(1- 2/\lambda_k\right) \|\vz^{k+1} - \vz^{k}\|_{U_k}^2.$
\item In Algorithm~\ref{alg:PDFBS}, we have
\begin{align*}
\kappa_u^k(\lambda_k) &\leq  \left(\rho- \frac{\varepsilon}{\beta\lambda_k}\right) \|\vz^{k+1} - \vz^{k}\|_{}^2 + 2\gamma_k\lambda_k\vg(\vx_\vg^k) - 2\gamma_k\lambda_k\vg(\vx_\vf^k).
\end{align*}
\item In Algorithm~\ref{alg:PDPRS}, we have $\kappa_u^k(\lambda_k) =  \left(1- 2/\lambda_k\right) \|\vz^{k+1} - \vz^{k}\|_{U}^2.$
\item In Algorithm~\ref{alg:PDFBF}, we have $\kappa_u^k(\lambda_k) \leq  0.$
\end{remunerate}
\end{proposition}
{\em Proof.}
Fix $k \in \vN$. To simplify notation, we drop the iteration index and denote $\vz := \vz^k, \vx_\vf := \vx_\vf^k, \vx_\vg := \vx_\vg^k, \vx_\vS := \vx_\vS^k, \vz^+ := \vz^{k+1}, \gamma := \gamma_k, \lambda := \lambda_k, U := U_k, $ and $\kappa_u := \kappa_u^k(\lambda_k)$ throughout this proof.

For PPA, FBS, and PRS, we note that the following identities hold: 
\begin{align}\label{defi:PDFBTR:eq:main}
\vz^{+} - \vz &= \lambda(\vx_\vf - \vx_\vg),
\end{align}
and there exists $w \in \vR$ such that
\begin{align}\label{defi:PDFBTR:eq:convex}
\vx_{\vS} &= w \vx_\vf +  (1-w) \vx_\vg.
\end{align}
Indeed, in PPA and FBS, $w = 1$ (see Section~\ref{sec:examplesofUS}).  In PRS, $w$ is a parameter of the algorithm, and Equations~\eqref{defi:PDFBTR:eq:convex} and~\eqref{defi:PDFBTR:eq:main} are shown in Lemma~\ref{lem:spunif}. Furthermore, Part~\ref{prop:basicprox:part:J} of Proposition~\ref{prop:basicprox} shows that in PPA and FBS,
\begin{align}\label{eq:FPRsubgradidentity}
\vx_\vf &= \vx_\vg - \gamma U^{-1} \left( \tnabla \vf(\vx_\vf) + \nabla \vg(\vx_\vg) + \vS\vx_{\vS}\right)
\end{align}
for a unique subgradient $\tnabla \vf(\vx_\vf) \in \partial \vf(\vx_\vf)$; see Lemma~\ref{lem:spunif} for the definition of $\tnabla \vf(\vx_\vf)$. 

Now we claim that in PPA, FBS, and PRS,
\begin{align*}
\kappa_u &= 2\dotp{\vx_\vf + \gamma U^{-1}(\tnabla \vg(\vx_{\vg}) + (1-w) \vS\vx_\vg) - \vz^{+}, \vz - \vz^{+}}_{U} - \|\vz^{+} - \vz\|_{U}^2\numberthis \label{lem:PDFBTRupperkeyterm:eq:main}
\end{align*}
where we make the identification $\tnabla \vg(\vx_\vg) = \nabla \vg(\vx_\vg)$ whenever $\vg$ is differentiable; see Lemma~\ref{lem:spunif} for the definition of $\tnabla \vg(\vx_\vg)$ in the PRS algorithm.  Because $\vx_\vS = \vx_\vg + w(\vx_\vf - \vx_\vg) = \vx_\vg + (w/\lambda)(\vz^{+} - \vz)$ and $\dotp{\vS\vx, \vx} = 0$ for all $\vx \in \vH$, we have the simplification:
\begin{align}\label{eq:xstoxg}
2\dotp{\vz - \vz^{+}, \gamma (1-w)\vS \vx_{\vS}} = 2\dotp{\vz - \vz^{+}, \gamma (1-w)\vS \vx_{\vg}}.
\end{align}
Therefore, 
\begin{align*}
\kappa_u &= - \|\vz^{+} - \vz\|_{U}^2 + 2\gamma \lambda \dotp{\vx_{\vf} - \vz^{+}, \tnabla \vf(\vx_{\vf})} \\
&+ 2\gamma \lambda  \dotp{\vx_{\vg} - \vz^{+}, \tnabla \vg(\vx_{\vg})} + 2\gamma \lambda   \dotp{\vx_{\vS}-\vz^{+}, \vS\vx_{\vS}} \\
&= - \|\vz^{+} - \vz\|_{U}^2 + 2\gamma\lambda \dotp{\vx_{\vf} - \vz^{+}, \tnabla \vf(\vx_{\vf})} \\
&+2\gamma\lambda \dotp{\vx_{\vg} -\vx_\vf, \tnabla \vg(\vx_{\vg})} + 2\gamma\lambda  \dotp{\vx_{\vf} - \vz^{+}, \tnabla \vg(\vx_{\vg})}\\
& +2\gamma\lambda \dotp{\vx_{\vS} -\vx_\vf, \vS \vx_{\vS}} + 2\gamma\lambda  \dotp{\vx_{\vf} - \vz^{+}, \vS\vx_{\vS}} \\
&= - \|\vz^{+} - \vz\|_{U}^2 + 2\gamma\lambda \dotp{\vx_{\vf} - \vz^{+},  \tnabla \vf(\vx_\vf) + \tnabla \vg(\vx_\vg) + \vS\vx_{\vS}} \\
&\stackrel{\eqref{defi:PDFBTR:eq:main}}{+}2\dotp{\vz - \vz^{+}, \gamma\tnabla \vg(\vx_{\vg})}  + 2\dotp{\vz - \vz^{+}, \gamma (1-w)\vS \vx_{\vS}} \\
&\stackrel{\eqref{defi:PDFBTR:eq:main}}{=}  -  \|\vz^{+} - \vz\|_{U}^2 + 2  \dotp{\vx_{\vf} - \vz^{+},  \vz - \vz^{+}}_{U}\\
&\stackrel{\eqref{eq:xstoxg}}{+}2\dotp{\vz - \vz^{+}, \gamma\tnabla \vg(\vx_{\vg}) +  \gamma (1-w)\vS \vx_{\vg}}  \\
&= 2\dotp{\vx_\vf + \gamma U^{-1}(\tnabla \vg(\vx_{\vg}) + (1-w) \vS\vx_\vg) - \vz^{+}, \vz - \vz^{+}}_{U} - \|\vz^{+} - \vz\|_{U}^2
\end{align*}
where the second to last equality uses Equation~\eqref{eq:FPRsubgradidentity} and the second to last ``$+$" also uses Equation~\eqref{defi:PDFBTR:eq:convex}.

Now we proceed with the specific cases:  In PPA and FBS, $w = 1$ and
\begin{align*}
\kappa_u &\stackrel{\eqref{lem:PDFBTRupperkeyterm:eq:main}}{=} 2\dotp{\vx_\vf + \gamma U^{-1}\nabla \vg(\vx_{\vg}) - \vz^{+}, \vz - \vz^{+}}_{U} - \|\vz^{+} - \vz\|_{U}^2 \\
&= 2\dotp{\vx_\vf - \vz^{+}, \vz - \vz^{+}}_{U} + 2\gamma\dotp{\nabla\vg(\vx_{\vg}), \vz - \vz^{+}} -  \|\vz^{+} - \vz\|_{U}^2\\
&=  2\left(1- \frac{1}{\lambda}\right) \|\vz^{+} - \vz\|_{U}^2 + 2\gamma \lambda\dotp{\nabla \vg(\vx_{\vg}), \vx_\vg - \vx_\vf} -  \|\vz^{+} - \vz\|_{U}^2 \numberthis\label{eq:ppaupperstopshort}\\
&\leq \left(1- \frac{2}{\lambda}\right) \|\vz^{+} - \vz\|_{U}^2 + 2\gamma\lambda\vg(\vx_\vg) -2\gamma\lambda\vg(\vx_\vf) + \frac{\gamma}{\lambda\beta} \|\vz^{+} - \vz\|^2
\end{align*}
where use the identity $\vx_\vf - \vz^+ = \left(1-(1/\lambda)\right)(\vz - \vz^+)$ on the third line, we use the identity $\vz^+ - \vz = \lambda(\vx_\vf - \vx_\vg)$ (Equation~\eqref{defi:PDFBTR:eq:main}) on the last two lines, and the last inequality follows from the Descent Theorem~\cite[Theorem 18.15(iii)]{bauschke2011convex}: $\dotp{\nabla \vg(\vx_\vg), \vx_\vg - \vx_\vf} \leq \vg(\vx_\vg) - \vg(\vx_\vf) + (1/(2\beta))\|\vx_\vg - \vx_\vf\|^2.$  In PPA $\vg \equiv 0$, so the Equation~\eqref{eq:ppaupperstopshort} implies the identity in Part~\ref{prop:upperkeyterms:part:PPA}. The inequality for FBS now follows by the above bound for $\kappa_u$, the bound $\gamma \leq 2\beta\rho - \varepsilon$, and 
\begin{align*}
\left(1- \frac{2}{\lambda}\right) \|\vz^{+} - \vz\|_{U}^2 + \frac{\gamma}{\lambda\beta} \|\vz^{+} - \vz\|^2 &\leq \left(\rho +  \frac{\gamma - 2\beta\rho}{\lambda\beta}\right)\|\vz^{+} - \vz\|^2
\end{align*}
where we use $\lambda \leq (4\beta\rho - \gamma)/2\beta\rho \leq 2$ and the lower bound $U \succcurlyeq \rho I_{\vH}$.

For relaxed PRS, we have
\begin{align*}
\vz^{+}  = \vz + \lambda\left(\vx_\vf - \vx_\vg\right) &= (1-\lambda) \vz + \lambda\left(\vx_\vf - \vx_\vg + \vz\right) \\
&= (1-\lambda) \vz + \lambda\left(\vx_\vf + \gamma U^{-1}\left(\tnabla \vg(\vx_\vg) + (1-w)\vS\vx_\vg\right)\right).
\end{align*}
Therefore, subtract $\lambda \vz^+ + (1-\lambda)\vz$ from both sides of the above equation, divide by $\lambda$, and use the identity in Equation~\eqref{lem:PDFBTRupperkeyterm:eq:main} to get
\begin{align*}
\kappa_u &=  2\dotp{\vx_\vf + \gamma U^{-1}\left(\tnabla \vg(\vx_{\vg}) + (1-w) \vS\vx_\vg\right) - \vz^{+}, \vz - \vz^{+}}_{U} - \|\vz^{+} - \vz\|_{U}^2 \\
&= 2\left(1 - \frac{1}{\lambda}\right) \|\vz^{+} - \vz\|^2_U- \|\vz^{+} - \vz\|_{U}^2 =\left(1 - \frac{2}{\lambda}\right) \|\vz^{+} - \vz\|^2_U.
\end{align*}

Finally, we prove the bound for the FBF algorithm: 
\begin{align*}
\kappa_u &\stackrel{\eqref{eq:PDupperkeyterm}}{=}  2\dotp{\vx_\vf - \vz^{+}, \gamma\tnabla \vf(\vx_\vf) +\gamma \nabla \vg(\vx_\vf) + \gamma \vS\vx_\vf} - \|\vz^{+} - \vz\|_{U}^2 \\
&= 2\dotp{\vx_\vf - \vz^{+}, \vz - \vz^{+}}_{U} - \|\vz^{+} - \vz\|_{U}^2  \stackrel{\eqref{eq:cosinerule}}{=} \|\vx_\vf - \vz^{+}\|_{U}^2 - \|\vx_\vf - \vz\|_{U}^2.
\end{align*}
Furthermore, the identity holds:
\begin{align*}
\vz^{+} - \vx_\vf &= \vz^{+} - \vz + \vz - \vx_\vf \\
&= -\gamma U^{-1} \left( \tnabla \vf(\vx_\vf) + \nabla \vg(\vx_\vf) + \vS\vx_{\vf}\right) + \gamma U^{-1} \left( \tnabla \vf(\vx_\vf) + \nabla \vg(\vz) + \vS\vz\right) \\
&= \gamma U^{-1}\left(\nabla \vg(\vz) + \vS\vz - \nabla \vg(\vx_\vf) -\vS\vx_{\vf} \right). \numberthis \label{eq:fbflipidentity}
\end{align*}
Note that the operator $\nabla \vg + \vS$ is $(1/\beta) + \|\vS\|$ Lipschitz. Thus, 
\begin{align*}
\|\vx_\vf - \vz^{+}\|^2_{U} - \|\vx_\vf - \vz\|^2_{U} &\stackrel{\eqref{eq:fbflipidentity}}{=} \gamma^2\|\nabla \vg(\vz) + \vS\vz - \nabla \vg(\vx_\vf) - \vS\vx_\vf\|^2_{U^{-1}} - \|\vx_\vf - \vz\|^2_{U}  \\
&\leq  \left(\frac{\gamma^2}{\rho}\left(\frac{1}{\beta} + \|\vS\|\right)^2 - \rho\right)\|\vx_\vf - \vz\|^2 \leq 0,
\end{align*}
where we use the following bound: for all $\vx \in \vH$, $\|\vx\|_{U^{-1}}^2 \leq (1/\rho)\|\vx\|^2$ (Lemma~\ref{lem:metricproperties}).\endproof


\section{Ergodic convergence}\label{sec:ergodic}

In this section, we prove an ergodic convergence rate for the pre-primal-dual gap. To this end, we recall the partial sum sequence $\Sigma_k  = \sum_{i=0}^k \gamma_i\lambda_i,$ and for every sequence of vectors $(\vx^j)_{j\geq 0} \subseteq \vH$, we define the ergodic sequence $\overline{\vx}^k = ({1}/{\Sigma_k}) \sum_{i=0}^k \gamma_i\lambda_i \vx^i$. For each algorithm, Theorem~\ref{thm:PDerg} (below) gives an ergodic sequence $(\overline{\vx}^j)_{j \in \vN}$ such that for all bounded subsets $D \subseteq \vH$, we have
\begin{align*}
\sup_{\vx \in D} \cG^{\mathrm{pre}}(\overline{\vx}^k, \overline{\vx}^k, \overline{\vx}^k; \vx) &= O\left(\frac{1+\sup_{\vx \in D} \|\vx\|^2}{\Sigma_k}\right).
\end{align*}
This bound is a generalization of the primal-dual gap bounds shown in~\cite{chambolle2011first,boct2013convergence,bo2014convergence,boct2014convergence}. See Section~\ref{sec:preprimaldualprimal} for several lower bounds of the pre-primal-dual gap.

Before we prove our ergodic rates, we need to prove a bound for PRS. Recall that we only analyze the PRS algorithm when the map $U_k \equiv U$ is fixed.  The following lemma will help us deduce the convergence rate of the PRS algorithm whenever $\vf$ or $\vg$ is Lipschitz (Part~\ref{thm:PDerg:part:PRS} of Theorem~\ref{thm:PDerg}).

\begin{lemma}\label{lemma:PDFBTR:sequencedistancePRS}
Suppose that $(\vz^j)_{j \in \vN}$ is generated by the relaxed PRS algorithm and that $\vz^\ast$ is a fixed-point of $\TPRS$ (see equation~\eqref{eq:TPRS}). Then the following ergodic bound holds: for all $k \in \vN$, we have
\begin{align}\label{eq:ergodicfprPRS}
\|\overline{\vx}_\vf^k - \overline{\vx}_\vg^k\|_U &\leq \frac{2\gamma\|\vz^0 -\vz^\ast\|_U}{\Sigma_k}.
\end{align} 
\end{lemma}
{\em Proof.}
Fix $k \in \vN$. The identity $\lambda_k(\vx_\vf^k - \vx_{\vg}^k) = \vz^{k+1} - \vz^k$ and the fact the sequence $(\|\vz^j - \vz^\ast\|_U)_{j \in \vN}$ is decreasing (Part~\ref{prop:bddsum:part:PDPRS} of Proposition~\ref{prop:bddsum}), show that
\begin{align*}
\left\|\overline{\vx}_\vf^k - \overline{\vx}_\vg^k\right\|_U = \left\|\frac{\gamma}{\Sigma_k}\sum_{i = 0}^k\lambda_i(\vx_\vf^i - \vx_\vg^i)\right\|_U = \frac{\gamma\| \vz^{k+1} - \vz^0\|_U}{\Sigma_k} &\leq \frac{\gamma \|\vz^{k+1} - \vz^\ast\|_U + \gamma \|\vz^0 - \vz^\ast\|_U}{\Sigma_k} \\
&\leq \frac{2\gamma\|\vz^0 - \vz^\ast\|_U}{\Sigma_k}. \qquad \endproof
\end{align*}

Lemma~\ref{lemma:PDFBTR:sequencedistancePRS} shows that the difference of splitting variables $\overline{\vx}_{\vf}^k - \overline{\vx}_\vg^k$ converges to zero with rate $O({1}/{\Sigma_k})$. Thus, if $\vf$ is Lipschitz continuous, then $|\vf(\overline{\vx}_\vf^k) - \vf(\overline{\vx}_\vg^k)| = O(1/\Sigma_k)$.

We are now ready to prove our main ergodic convergence results.

\begin{theorem}[Ergodic convergence of the unifying scheme]\label{thm:PDerg}
Suppose that the sequence $(\vz^{j})_{j \in \vN}$ is generated by Algorithm~\ref{alg:US}, and suppose that Assumption~\ref{assump:variablemetric} holds. Then for all $\vx \in \vH$ and all $k \in \vN$, we have the following bounds:
\begin{remunerate}
\item \label{thm:PDerg:part:PPA} \textbf{Ergodic convergence of PPA:} Let $\vz^\ast \in \zer(\partial \vf + \vS)$.  Then in Algorithm~\ref{alg:PPA}, we have
\begin{align*}
\cG^{\mathrm{pre}}(\overline{\vx}_{\vf}^k,\overline{\vx}_{\vf}^k, \overline{\vx}_\vf^k; \vx) &\leq \frac{ \|\vz^0 - \vx\|^2_{U_0} + 2\eta_{\mathrm{p}}\eta_{\mathrm{s}}\|\vz^0 - \vz^\ast\|_{U_0}^2 + 2\mu \eta_{\mathrm{s}} \|\vz^\ast - \vx\|^2}{2\Sigma_k}.
\end{align*}
\item \label{thm:PDerg:part:FBS} \textbf{Ergodic convergence of FBS:} Let $\vz^\ast \in \zer(\partial \vf + \nabla \vg + \vS)$, and let $\overline{\lambda} = \sup_{j \in \vN} \lambda_j$. Then in Algorithm~\ref{alg:PDFBS}, we have the bounds $0 < \inf_{j \in \vN} \lambda_j \leq \overline{\lambda} \leq 2$ and $\inf_{j \in \vN} (1-\alpha_j\lambda_j)/(\alpha_j\lambda_j) > 0$,  and 
\begin{align*}
&\cG^{\mathrm{pre}}(\overline{\vx}_{\vf}^k, \overline{\vx}_{\vf}^k, \overline{\vx}_\vf^k; \vx) \\
&\leq \frac{\left(\|\vz^0 - \vx\|^2_{U_0} + \left(2\eta_{\mathrm{p}}\eta_{\mathrm{s}} + \frac{\left(1 + \eta_{\mathrm{p}}\eta_{\mathrm{s}}\right)\max\left\{\rho- \varepsilon/(\beta\overline{\lambda}), 0\right\}}{\rho\inf_{j \in \vN} (1-\alpha_j\lambda_j)/(\alpha_j\lambda_j)}\right)\|\vz^0 - \vz^\ast\|_{U_0}^2 + 2\mu \eta_{\mathrm{s}} \|\vz^\ast - \vx\|^2\right)}{2\Sigma_k}.
\end{align*}
\item \label{thm:PDerg:part:PRS} \textbf{Ergodic convergence of PRS:} Let $\vz^\ast$ be a fixed point of $\TPRS$. Suppose that $\vf$ (respectively $\vg$) is $L$-Lipschitz, let $\vx^k := \vx_\vg^k$ (respectively $\vx^k := \vx_\vf^k$), and let $\hat{w} = w$ (respectively $\hat{w} =1-w$). Then in Algorithm~\ref{alg:PDPRS}, we have
\begin{align*}
\cG^{\mathrm{pre}}(\overline{\vx}^k,\overline{\vx}^k, \overline{\vx}^k; \vx) &\leq \frac{\|\vz^0 - \vx\|_U^2 + 4(\gamma/\sqrt{\rho}) (L + |\hat{w}|\|\vS\|\|\vx\|)\|\vz^0 - \vz^\ast\|_U }{2\Sigma_k}. 
\end{align*}
\item \label{thm:PDerg:part:FBF} \textbf{Ergodic convergence of FBF:} Let $\vz^\ast \in \zer(\partial \vf + \nabla \vg + \vS)$. Then in Algorithm~\ref{alg:PDFBF}, we have
\begin{align*}
\cG^{\mathrm{pre}}(\overline{\vx}_{\vf}^k,\overline{\vx}_\vf^k, \overline{\vx}_\vf^k; \vx) &\leq  \frac{\|\vz^0 - \vx\|^2_{U_0} + 2\eta_{\mathrm{p}}\eta_{\mathrm{s}}\|\vz^0 - \vz^\ast\|_{U_0}^2 + 2\mu \eta_{\mathrm{s}} \|\vz^\ast - \vx\|^2}{2\Sigma_k}.
\end{align*}
\end{remunerate}
\end{theorem}
{\em Proof.}
Fix $k \in \vN$. For any sequence of points $(\vz^j)_{j \in \vN}\subseteq \vH$ and any point $\vz^\ast\in \vH$ such that $\|\vz^{i+1} - \vz^\ast\|_{U_{i+1}}^2 \leq (1+\eta_i)\|\vz^i - \vz^\ast\|_{U_i}^2$ for all $i \in \vN$, we have $\|\vz^i - \vz^\ast\|^2_{U_{i}} \leq \left(\prod_{i=0}^\infty (1+\eta_i)\right)\|\vz^0 - \vz^\ast\|_{U_0}^2.$
Therefore, by the convexity of $\|\cdot \|_{U_{i}}^2$ for all $i \in \vN$, and by the inequality $-\|\vx\|_{U_{i}} \leq -(1/(1+\eta_i))\|\vx\|_{U_{i+1}}$  for all $\vx \in \vH$ and $i \in \vN$, we have 
\begin{align*}
&\sum_{i=0}^k\left( \|\vz^{i} - \vx\|_{U_i}^2 - \|\vz^{i+1} - \vx\|_{U_i}^2\right) \\
&\leq \|\vz^0 - \vx\|_{U_0}^2 + \sum_{i=0}^k \left(\|\vz^{i+1} - \vx\|^2_{U_{i+1}} - \|\vz^{i+1} - \vx\|^2_{U_{i}}\right)\\
&\leq \|\vz^0 - \vx\|^2_{U_0} + \sum_{i=0}^k \frac{\eta_i}{1+\eta_i}\|\vz^{i+1}  -\vx\|^2_{U_{i+1}} \\
&\leq \|\vz^0 - \vx\|^2_{U_0} + 2\sum_{i=0}^k \eta_i\left(\|\vz^{i+1}  -\vz^\ast\|^2_{U_{i+1}} + \|\vz^\ast - \vx\|^2_{U_{i+1}}\right) \\
&\leq  \|\vz^0 - \vx\|^2_{U_0} + \left( 2\left(\prod_{i=0}^\infty (1+\eta_i)\right)\sum_{i=0}^\infty \eta_i\right) \|\vz^0 - \vz^\ast\|^2_{U_0} + 2\mu\left(\sum_{i=0}^\infty \eta_i\right) \|\vz^\ast - \vx\|^2 \\
&=  \|\vz^0 - \vx\|^2_{U_0} + 2\eta_{\mathrm{p}}\eta_{\mathrm{s}}\|\vz^0 - \vz^\ast\|_{U_0}^2 + 2\mu \eta_{\mathrm{s}} \|\vz^\ast - \vx\|^2. \numberthis \label{thm:PDerg:eq:normsumbound}
\end{align*}
We will use Equation~\eqref{thm:PDerg:eq:normsumbound} to produce bounds for all of the variable metric methods.

Part~\ref{thm:PDerg:part:PPA}: This follows from the Jensen's inequality, Proposition~\ref{prop:upperkeyterms} ($\kappa_{u}^i =\left(1- {2}/{\lambda_i}\right) \|\vz^{i+1} - \vz^{i}\|_{U_i}^2  \leq 0$), and the fundamental inequality ($\cG^{\mathrm{pre}}$ does not depend on its second input):
\begin{align*}
\cG^{\mathrm{pre}}(\overline{\vx}_{\vf}^k,\overline{\vx}_{\vf}^k, \overline{\vx}_\vf^k; \vx) &\leq \frac{1}{\Sigma_k}\sum_{i=0}^k \gamma_i\lambda_i\cG^{\mathrm{pre}}({\vx}_{\vf}^i,\vx_{\vf}^i, {\vx}_\vf^i; \vx) \\
&\stackrel{\eqref{eq:PDupper}}{\leq} \frac{1}{2\Sigma_k}\sum_{i=0}^k\left(\kappa_u^i + \|\vz^{i} - \vx\|_{U_i}^2 - \|\vz^{i+1} - \vx\|_{U_i}^2\right) \\
&\stackrel{\eqref{thm:PDerg:eq:normsumbound}}{\leq} \frac{1}{2\Sigma_k}\left( \|\vz^0 - \vx\|^2_{U_0} + 2\eta_{\mathrm{p}}\eta_{\mathrm{s}}\|\vz^0 - \vz^\ast\|_{U_0}^2 + 2\mu \eta_{\mathrm{s}} \|\vz^\ast - \vx\|^2\right).
\end{align*}

Part~\ref{thm:PDerg:part:FBS}: We have the following bound from Proposition~\ref{prop:bddsum}:
\begin{align*}
 \sum_{i=0}^k \left(\rho- \frac{\varepsilon}{\beta\lambda_i}\right) \|\vz^{i+1} - \vz^{i}\|^2 &\leq  \frac{\max\left\{\rho- \varepsilon/(\beta\overline{\lambda}), 0\right\}}{\rho\inf_{j \in \vN} (1-\alpha_j\lambda_j)/(\alpha_j\lambda_j)}\left(1 + \eta_{\mathrm{p}}\eta_{\mathrm{s}}\right)\|\vz^0 - \vz^\ast\|_{U_0}^2. \numberthis \label{eq:FBSergodicFPR}
\end{align*}
Thus, the bound follows from Jensen's inequality, Proposition~\ref{prop:upperkeyterms} {($\kappa_{u}^i \leq \left(\rho- {\varepsilon}/{(\beta\lambda_i)}\right) \|\vz^{i+1} - \vz^{i}\|^2 + 2\gamma_i\lambda_i\vg(\vx_\vg^i) - 2\gamma_i\lambda_i\vg(\vx_\vf^i)$)}, and the fundamental inequality:
\begin{align*}
&\cG^{\mathrm{pre}}(\overline{\vx}_{\vf}^k, \overline{\vx}_{\vf}^k, \overline{\vx}_\vf^k; \vx) \leq \frac{1}{\Sigma_k}\sum_{i=0}^k \gamma_i \lambda_i\cG^{\mathrm{pre}}({\vx}_{\vf}^i, {\vx}_{\vf}^i, {\vx}_\vf^i; \vx)  \\
&= \frac{1}{\Sigma_k}\sum_{i=0}^k \left(\gamma_i\lambda_i\cG^{\mathrm{pre}}({\vx}_{\vf}^i, {\vx}_{\vg}^i, {\vx}_\vf^i;\vx) + \gamma_i\lambda_i\vg(\vx_\vf^i) - \gamma_i\lambda_i\vg(\vx_\vg^i)\right) \\
&\stackrel{\eqref{eq:PDupper}}{\leq} \frac{1}{2\Sigma_k} \sum_{i=0}^k\left( \kappa_u^i  + \|\vz^{i} - \vx\|_{U_i}^2 - \|\vz^{i+1} - \vx\|_{U_i}^2 + 2\gamma_i\lambda_i\vg(\vx_\vf^i) - 2\gamma_i\lambda_i\vg(\vx_\vg^i)\right) \\
&\stackrel{\eqref{thm:PDerg:eq:normsumbound}}{\leq} \frac{1}{2\Sigma_k}\left( \sum_{i=0}^k \left(\rho- \frac{\varepsilon}{\beta\lambda_i}\right) \|\vz^{i+1} - \vz^{i}\|^2\right) \\
&+  \frac{1}{2\Sigma_k}\left( \|\vz^0 - \vx\|^2_{U_0} + 2\eta_{\mathrm{p}}\eta_{\mathrm{s}}\|\vz^0 - \vz^\ast\|_{U_0}^2 + 2\mu \eta_{\mathrm{s}} \|\vz^\ast - \vx\|^2\right) \\
&\stackrel{\eqref{eq:FBSergodicFPR}}{\leq}  \frac{\left(\|\vz^0 - \vx\|^2_{U_0} + \left(2\eta_{\mathrm{p}}\eta_{\mathrm{s}} + \frac{\left(1 + \eta_{\mathrm{p}}\eta_{\mathrm{s}}\right)\max\left\{\rho- \varepsilon/(\beta\overline{\lambda}), 0\right\}}{\rho\inf_{j \in \vN} (1-\alpha_j\lambda_j)/(\alpha_j\lambda_j)}\right)\|\vz^0 - \vz^\ast\|_{U_0}^2 + 2\mu \eta_{\mathrm{s}} \|\vz^\ast - \vx\|^2\right)}{2\Sigma_k}.
\end{align*}

Part~\ref{thm:PDerg:part:PRS}: We prove the result when $\vf$ is Lipschitz; the other case is symmetric.  This follows from the Jensen's inequality, Proposition~\ref{prop:upperkeyterms}{ ($\kappa_{u}^i =\left(1- {2}/{\lambda_i}\right) \|\vz^{i+1} - \vz^{i}\|_{U}^2  \leq 0$)}, the fundamental inequality, and the identity $\overline{\vx}_\vg^k - \overline{\vx}_\vS^k = w(\overline{\vx}_\vg^k - \overline{\vx}_\vf^k)$ (follows by averaging identities found in Part~\ref{lem:spunif:part:PDPRS} of Lemma~\ref{lem:spunif}):
\begin{align*}
\cG^{\mathrm{pre}}(\overline{\vx}_{\vg}^k,\overline{\vx}_{\vg}^k, \overline{\vx}_{\vg}^k; \vx) &= \cG^{\mathrm{pre}}(\overline{\vx}_{\vf}^k,\overline{\vx}_{\vg}^k, \overline{\vx}_{\vS}^k; \vx) + \vf(\overline{\vx}_\vg^k) - \vf(\overline{\vx}_\vf^k) + \dotp{\vS(\overline{\vx}_{\vg}^k - \overline{\vx}_{\vS}^k), -\vx} \\
 &\leq \frac{1}{\Sigma_k}\sum_{i=0}^k \gamma\lambda_i\cG^{\mathrm{pre}}({\vx}_{\vf}^i,{\vx}_{\vg}^i, {\vx}_{\vS}^i;\vx) \\
 &+ \vf(\overline{\vx}_\vg^k) - \vf(\overline{\vx}_\vf^k) + \dotp{\vS(\overline{\vx}_{\vg}^k - \overline{\vx}_{\vS}^k), -\vx}  \\
&\stackrel{\eqref{eq:PDupper}}{\leq}   \frac{1}{2\Sigma_k}\sum_{i=0}^k\left(\kappa_u^i + \|\vz^{i} - \vx\|_{U}^2 - \|\vz^{i+1} - \vx\|_{U}^2\right) \\
&+ L\|\overline{\vx}_{\vg}^k - \overline{\vx}_\vf^k\| + \|\vS\|\|\overline{\vx}_{\vg}^k - \overline{\vx}_{\vS}^k\|\|\vx\| \\
&\stackrel{\eqref{eq:ergodicfprPRS}}{\leq} \frac{\|\vz^0 - \vx\|_U^2 + 4(\gamma/\sqrt{\rho}) (L + |w|\|\vS\|\|\vx\|)\|\vz^0 - \vz^\ast\|_U }{2\Sigma_k} .
\end{align*}

Part~\ref{thm:PDerg:part:FBF}: This follows from the Jensen's inequality, Proposition~\ref{prop:upperkeyterms} ($\kappa_{u}^i \leq 0$), and the fundamental inequality:
\begin{align*}
\cG^{\mathrm{pre}}(\overline{\vx}_{\vf}^k,\overline{\vx}_{\vf}^k, \overline{\vx}_\vf^k; \vx) &\leq \frac{1}{\Sigma_k}\sum_{i=0}^k \gamma_i\lambda_i\cG^{\mathrm{pre}}({\vx}_{\vf}^i,{\vx}_{\vf}^i, {\vx}_\vf^i; \vx) \\
&\stackrel{\eqref{eq:PDupper}}{\leq} \frac{1}{2\Sigma_k}\sum_{i=0}^k\left(\kappa_u^i + \|\vz^{i} - \vx\|_{U_i}^2 - \|\vz^{i+1} - \vx\|_{U_i}^2\right) \\
&\stackrel{\eqref{thm:PDerg:eq:normsumbound}}{\leq} \frac{1}{2\Sigma_k}\left(\|\vz^0 - \vx\|^2_{U_0} + 2\eta_{\mathrm{p}}\eta_{\mathrm{s}}\|\vz^0 - \vz^\ast\|_{U_0}^2 + 2\mu \eta_{\mathrm{s}} \|\vz^\ast - \vx\|^2\right). \qquad \endproof
\end{align*}

\begin{remark}
In general, the $O(1/(k+1))$ convergence rates in Theorem~\ref{thm:PDerg} are the best PPA, FBS, and PRS obtain for $(\overline{\vx}_\vf^j)_{j \in \vN}$ and $(\overline{\vx}_{\vg}^j)_{j \in \vN}$ \cite[Proposition 8]{davis2014convergence}.
\end{remark}

\section{Nonergodic convergence}\label{sec:nonergodic}

In this section we deduce nonergodic convergence rates for PPA, FBS and PRS under the following assumption:

\begin{assump}\label{assump:nonergodic}
For all nonergodic convergence results, we assume $(U_j)_{j \in \vN}$ and $(\gamma_j)_{j \in \vN}$ are constant sequences. 
\end{assump}

For PPA, FBS, and PRS, Theorem~\ref{thm:nonergodic} (below) produces a natural sequence $(\vx^j)_{j \in \vN}$ such that for all bounded subsets $D \subseteq \vH$, we have
\begin{align*}
\sup_{\vx \in D} \cG^{\mathrm{pre}}(\vx^k, \vx^k, \vx^k; \vx) &= o\left(\frac{1+ \sup_{\vx \in D} \|\vx\|_U}{\sqrt{k+1}}\right).
\end{align*}
To the best of our knowledge, the rate of convergence for the nonergodic primal-dual gap generated by the class of algorithms we study has never appeared in the literature.

Nonergodic iterates tend to share structural properties, such as sparsity or low rank, with the solution of the problem. In some cases, the ergodic iterates generated in Section~\ref{sec:ergodic} ``average out" structural properties of the nonergodic iterates. Thus, although the ergodic iterates may be ``closer" to the solution, they are often poorer partial solutions than the nonergodic iterates. The results of this section provide worst-case theoretical guarantees on the quality of the nonergodic iterates in order to justify their use in practical applications.  

In our analysis, we use the following result (see also~\cite{doi:10.1137/130940402} for similar little-$o$ and big-$O$ convergence rates): 

\begin{theorem}[{\cite[Theorem 1]{davis2014convergence}}]\label{thm:FPRrate}
Let $\alpha \in (0, 1)$, let $\rho > 0$, let $U \in \cS_\rho(\vH)$, and let $(\lambda_j)_{j \in \vN} \subseteq (0, 1/\alpha)$. Suppose that $T : \vH \rightarrow \vH$ is an $\alpha$-averaged operator in the norm $\|\cdot\|_U$. Let $\vz^\ast$ be a fixed point of $T$, let $\vz^0 \in \vH$, let $\tau_k := (1-\alpha\lambda_k)\lambda_k/\alpha$ for all $k \in \vN$, suppose that $\underline{\tau} := \inf_{j \in \vN} \tau_j > 0$, and suppose that $(\vz^j)_{j \in \vN}$ is generated by the following iteration: for all $k \in \vN$, let
\begin{align}\label{eq:kmunnormalized}
\vz^{k+1} := T_{\lambda_k}(\vz^k).
\end{align}
Then for all $k \in \vN$, we have 
\begin{align}\label{eq:FPRrates}
\|T\vz^k - \vz^k\|_U^2 \leq \frac{\|\vz^0 - \vz^\ast\|_U^2}{\underline{\tau}(k+1)} && \mathrm{and} && \|T\vz^k - \vz^k\|_U^2 = o\left(\frac{1}{k+1}\right).
\end{align}
\end{theorem}

Throughout this section, $T$ will always denote an $\alpha$-averaged mapping in the norm $\|\cdot \|_U$.  Recall that for $\lambda \in (0, 1/\alpha)$, $T_{\lambda}$ is $\alpha\lambda$-averaged (see Proposition~\ref{prop:basicprox}), so 
\begin{align*}
\|T_{\lambda}\vz^k - \vz^\ast\|_U^2 &\stackrel{\eqref{eq:avgdecrease}}{\leq} \|\vz^k - \vz^\ast\|_U^2 - \frac{1- \alpha\lambda}{\alpha\lambda}\|T_{\lambda} \vz^k - \vz^k\|_U^2\numberthis\label{eq:fejer2}
\end{align*}
for all $k \in \vN$, and any fixed-point $\vz^\ast$ of $T$. Note that Equation~\eqref{eq:fejer2} also holds when $\alpha\lambda = 1$ (see Proposition~\ref{prop:basicprox}).  Equation~\eqref{eq:fejer2} shows that $T_\lambda \vz^k$ is at least as close to $\vz^\ast$ as $\vz^k$ is. This fact will be useful in the proof of Theorem~\ref{thm:nonergodic} below.

In the following theorem, we will deduce little-$o$ and big-$O$ convergence rates.  Because the pre-primal-dual gap can be negative, we slightly abuse notation: given a point $\vx \in \vH$, a (not necessarily positive) sequence $(a_j)_{j \in \vN}$ satisfies $a_k = o((1+\|\vx\|_U)/\sqrt{k+1})$ provided that there exists a nonnegative sequence $(b_j)_{j \in \vN}$ such that $b_k = o((1+\|\vx\|_U)/\sqrt{k+1})$ and $a_k = O(b_k)$. Note that we do not measure $|a_k|$ because our only goal is to ensure that the sequence $(a_j)_{j \in \vN}$ is eventually nonpositive.

\begin{theorem}\label{thm:nonergodic}
Suppose that Assumption~\ref{assump:nonergodic} holds, let $U \in \cS_\rho(\vH)$ denote the common metric inducing map, and let $\gamma \in \vR_{++}$ denote the common stepsize parameter. Then each method is a special case of Iteration~\eqref{eq:kmunnormalized}. For each method, assume that $\underline{\tau} > 0$ (See Theorem~\ref{thm:FPRrate}). Then for all $k \in \vN$ and all $\vx \in \vH$, the following hold:
\begin{remunerate}
\item\label{thm:nonergodic:part:PPA} \textbf{Nonergodic convergence of PPA:} Let $\vz^\ast \in \zer(\partial \vf  + \vS)$. Then in Algorithm~\ref{alg:PPA}, we have $\alpha = 1/2$ and $T = J_{U^{-1}(\partial \vf + \vS)} $, 
\begin{align*}
\cG^{\mathrm{pre}}({\vx}_{\vf}^k,{\vx}_{\vf}^k, {\vx}_\vf^k; \vx) &\leq \frac{\left(\|\vz^0 - \vz^\ast\|_U + \|\vz^\ast - \vx\|_U\right) \|\vz^0 - \vz^\ast\|_U}{\gamma\sqrt{\underline{\tau}(k+1)}},
\end{align*}
and $\cG^{\mathrm{pre}}({\vx}_{\vf}^k,\vx_\vf^k, {\vx}_\vf^k; \vx) = o\left((1+\|\vx\|_U)/\sqrt{k+1}\right)$.
\item \label{thm:nonergodic:part:FBS}\textbf{Nonergodic convergence of FBS:}  Let $\vz^\ast \in \zer(\partial \vf +  \nabla \vg + \vS)$. Then in Algorithm~\ref{alg:PDFBS}, we have $\alpha = \alpha_{\gamma, \rho}$ (Equation~\eqref{eq:alphafbs}) and $T = T_{\mathrm{FBS}}^{U, \gamma}$ (Equation~\eqref{eq:forwardbackwardoperator}),
\begin{align*}
\cG^{\mathrm{pre}}({\vx}_{\vf}^k,{\vx}_{\vf}^k, {\vx}_\vf^k; \vx) &\leq \frac{\left(\|\vz^0 - \vz^\ast\|_U + \|\vz^\ast - \vx\|_U\right) \|\vz^0 - \vz^\ast\|_U}{\gamma\sqrt{\underline{\tau}(k+1)}},
\end{align*}
 and $\cG^{\mathrm{pre}}({\vx}_{\vf}^k,{\vx}_{\vf}^k, {\vx}_\vf^k; \vx) = o\left((1+\|\vx\|_U)/\sqrt{k+1}\right)$.
\item \label{thm:nonergodic:part:PRS} \textbf{Nonergodic convergence of PRS:} Let $\vz^\ast$ be a fixed point of $\TPRS$ (Equation~\eqref{eq:TPRS}). Then in Algorithm~\ref{alg:PDPRS}, we have  $\alpha = 1/2$ and $T =  (\TPRS)_{1/2}$ (Equation~\eqref{eq:DRSoperator}). In addition, suppose that $\vf$ (respectively $\vg$) is $L$-Lipschitz, let $\vx^k := \vx_\vg^k$ (respectively $\vx^k := \vx_\vf^k$), and let $\hat{w} = w$ (respectively $\hat{w} =1-w$). Then
\begin{align*}
\cG^{\mathrm{pre}}({\vx}^k,{\vx}^k, {\vx}^k; \vx) &\leq \frac{\left(\|\vz^0 - \vz^\ast\|_U + \|\vz^\ast - \vx\|_U + (\gamma/\sqrt{\rho})(L + |\hat{w}|\|\vS\| \|\vx\|)\right) \|\vz^0 - \vz^\ast\|_U}{\gamma\sqrt{\underline{\tau}(k+1)}},
\end{align*}
and $\cG^{\mathrm{pre}}({\vx}^k,{\vx}^k, {\vx}^k; \vx) = o\left((1+\|\vx\|_U)/\sqrt{k+1}\right).$
\end{remunerate}
\end{theorem}
{\em Proof.}
Fix $k \in \vN$. In all of the following proofs, we will bound the pre-primal-dual gap by a quantity involving $\|T\vz^k - \vz^k\|_U$.  Then the big-$O$ and little-$o$ convergence rates follow directly from Theorem~\ref{thm:FPRrate}. In addition, we will use Equation~\eqref{eq:fejer2} and the independence of $\vx_\vf^k, \vx_\vg^k,$ and $\vx_\vS^k$ from $\lambda_k$ to tighten our upper bounds.  To this end, we will denote $\vz_\lambda := T_{\lambda}(\vz^k)$ (see Equation~\eqref{eq:averagednotation}) and let $C = (0, 1/\alpha]$ where $\alpha$ is averagedness coefficient of $T$. Note that $T_\lambda$ is nonexpansive for all $\lambda \in C$ (see Part~\ref{prop:basicprox:part:wider} of Proposition~\ref{prop:basicprox}). Also note that for $\lambda \in C$, we have $(1/\lambda)(\vz_{\lambda} - \vz^k )= T\vz^k - \vz^k$ and $\|\vz_{\lambda} - \vz^\ast\|_U \leq \|\vz^k - \vz^\ast\|_U \leq \|\vz^0 - \vz^\ast\|_U$ by Equation~\eqref{eq:fejer2} and the monotonicity of $(\|\vz^j - \vz^\ast\|_U)_{j \in \vN}$ (Proposition~\ref{prop:bddsum}). Thus, $\|\vz_{\lambda} - \vx\|_U \leq \|\vz^0 - \vz^\ast\|_U + \|\vz^\ast - \vx\|_U.$ Therefore, for all $\lambda \in (0, 1/\alpha]$, we have
\begin{align}\label{eq:independentFPR}
\frac{\dotp{\vz^k - \vz_{\lambda}, \vz_{\lambda} - \vx}_U}{\lambda} &\leq \|T \vz^k - \vz^k\|_U\|\vz_{\lambda} - \vx\|_U \stackrel{\eqref{eq:FPRrates}}{\leq} \frac{\left(\|\vz^0 - \vz^\ast\|_U + \|\vz^\ast - \vx\|_U\right) \|\vz^0 - \vz^\ast\|_U}{\sqrt{\underline{\tau}(k+1)}}.
\end{align}
Note that the upper key term identities (Proposition~\ref{prop:upperkeyterms}) and the fundamental inequality (Proposition~\ref{prop:PDupper}) continue to hold when $\vz^{k+1}$ is replaced by $\vz_\lambda$. Thus, in each of the cases below, we will minimize the fundamental inequality over all $\lambda \in C$.

Part~\ref{thm:nonergodic:part:PPA}: Proposition~\ref{prop:upperkeyterms} shows that $\kappa_{u}^k(\lambda) = \left(1-{2}/{\lambda}\right)\|\vz_\lambda - \vz^k\|_{U}^2.$ Thus, the fundamental inequality, the cosine rule, and the identity $C = (0, 2]$ show ($\cG^{\mathrm{pre}}$ does not depend on its second input)
\begin{align*}
\cG^{\mathrm{pre}}({\vx}_{\vf}^k,{\vx}_{\vf}^k, {\vx}_\vf^k; \vx) &\leq \inf_{\lambda \in C} \frac{1}{2\gamma\lambda}\left(\left(1-\frac{2}{\lambda}\right)\|\vz_\lambda - \vz^k\|_{U}^2 + \|\vz^{k} - \vx\|_{U}^2 - \|\vz_\lambda - \vx\|_{U}^2\right) \\
&\stackrel{\eqref{eq:cosinerule}}{=}\inf_{\lambda \in C} \frac{1}{2\gamma\lambda}\left(2 \dotp{\vz^k- \vz_\lambda, \vz_\lambda - \vx}_U + 2\left(1-\frac{1}{\lambda}\right)\|\vz_\lambda - \vz^k\|^2_U \right) \\
&\leq \frac{1}{\gamma}\dotp{\vz^k - \vz_1, \vz_1 - \vx}_U \stackrel{\eqref{eq:independentFPR}}{\leq} \frac{\left(\|\vz^0 - \vz^\ast\|_U + \|\vz^\ast - \vx\|_U\right) \|\vz^0 - \vz^\ast\|_U}{\gamma\sqrt{\underline{\tau}(k+1)}}.
\end{align*}

Part~\ref{thm:nonergodic:part:FBS}: First choose $\widetilde{\lambda} \in C$ small enough that $ \rho + \mu- {\varepsilon}/{(\beta\widetilde{\lambda})} \leq 0$. Now recall that Proposition~\ref{prop:upperkeyterms} proves the following inequality: $\kappa_{u}^k(\lambda) \leq \left(\rho- {\varepsilon}/({\beta\lambda})\right) \|\vz^{k+1} - \vz^{k}\|^2 + 2\gamma\lambda\vg(\vx_\vg^k) - 2\gamma\lambda\vg(\vx_\vf^k).$ Thus, the fundamental inequality, the cosine rule, and the identity $C = (0, 1/\alpha]$ show
\begin{align*}
&\cG^{\mathrm{pre}}({\vx}_{\vf}^k,{\vx}_{\vf}^k, {\vx}_\vf^k; \vx) = \cG^{\mathrm{pre}}({\vx}_{\vf}^k,{\vx}_{\vg}^k, {\vx}_\vf^k; \vx) + \vg(\vx_\vf^k)  - \vg(\vx_\vg^k) \\
&\leq \inf_{\lambda\in C}\frac{1}{2\gamma\lambda} \left( 2\gamma\lambda\vg(\vx_\vf^k)  - 2\gamma\lambda\vg(\vx_\vg^k) + \kappa_u^k(\lambda) + \|\vz^{k} - \vx\|_{U}^2 - \|\vz_\lambda - \vx\|_{U}^2\right) \\
&\leq \inf_{\lambda \in C} \frac{1}{2\gamma\lambda}\left( \left(\rho- \frac{\varepsilon}{\beta\lambda}\right) \|\vz_\lambda - \vz^{k}\|^2 + \|\vz^{k} - \vx\|_{U}^2 - \|\vz_\lambda - \vx\|_{U}^2\right) \\
&\stackrel{\eqref{eq:cosinerule}}{=} \inf_{\lambda \in C}\frac{1}{2\gamma\lambda}\left(2 \dotp{\vz^k - \vz_\lambda, \vz_\lambda - \vx}_U + \|\vz_\lambda - \vz^k\|^2_U + \left(\rho- \frac{\varepsilon}{\beta\lambda}\right) \|\vz_\lambda - \vz^{k}\|^2 \right)\\
&\leq \inf_{\lambda \in C}\frac{1}{2\gamma\lambda}\left(2 \dotp{\vz^k - \vz_\lambda , \vz_\lambda - \vx}_U + \left( (\rho+\mu)- \frac{\varepsilon}{\beta\lambda}\right) \|\vz_\lambda - \vz^{k}\|^2 \right) \\
&\leq \frac{1}{\gamma \widetilde{\lambda}} \dotp{\vz^k - \vz_{\widetilde{\lambda}}, \vz_{\widetilde{\lambda}} - \vx}_U \stackrel{\eqref{eq:independentFPR}}{\leq} \frac{\left(\|\vz^0 - \vz^\ast\|_U + \|\vz^\ast - \vx\|_U\right) \|\vz^0 - \vz^\ast\|_U}{\gamma\sqrt{\underline{\tau}(k+1)}}.
\end{align*}

Part~\ref{thm:nonergodic:part:PRS}: We prove the result in the case that $\vf$ is Lipschitz because the other case is symmetric. Proposition~\ref{prop:upperkeyterms} proves the following identity: $\kappa_{u}^k(\lambda) = \left(1-{2}/{\lambda}\right)\|\vz_\lambda - \vz^k\|_{U}^2.$ Thus, the fundamental inequality, the cosine rule, and the identities $\vx_\vf^k - \vx_\vg^k = (1/\lambda)(\vz_\lambda - \vz^k) = T\vz^k - \vz^k$,  $\vx_\vg^k - \vx_\vS^k = w(\vx_\vg^k - \vx_\vf^k)$, and  $C = (0, 2]$ show
\begin{align*}
&\cG^{\mathrm{pre}}(\vx_\vg^k, \vx_\vg^k, \vx_\vg^k; \vx) \\
&\leq \cG^{\mathrm{pre}}(\vx_\vf^k,\vx_\vg^k, \vx_\vS^k; \vx) + \vf(\vx_\vg^k) - \vf(\vx_\vf^k) + \dotp{\vS(\vx_{\vg}^k - \vx_{\vS}^k), -\vx} \\
&\leq \inf_{\lambda \in C}\frac{1}{2\gamma\lambda}\left(\left(1-\frac{2}{\lambda}\right)\|\vz_\lambda - \vz^k\|_{U}^2 + \|\vz^{k} - \vx\|_{U}^2 - \|\vz_\lambda - \vx\|_{U}^2\right) \\
&+ L\|\vx_\vg^k - \vx_\vf^k\| + |w|\|\vS\|\|\vx_\vg^k - \vx_\vf^k\|\|\vx\| \\
&\stackrel{\eqref{eq:cosinerule}}{=}\inf_{\lambda \in C} \frac{1}{2\gamma\lambda}\left(2 \dotp{\vz^k - \vz_\lambda, \vz_\lambda - \vx}_U + 2\left(1-\frac{1}{\lambda}\right)\|\vz_\lambda - \vz^k\|^2_U \right) \\
&+ L\|\vx_\vg^k - \vx_\vf^k\| + |w|\|\vS\|\|\vx_\vg^k - \vx_\vf^k\|\|\vx\| \\
&\leq \frac{1}{\gamma}\dotp{\vz^k - \vz_1, \vz_1 - \vx}_U + L\|\vx_\vg^k - \vx_\vf^k\| + |w|\|\vS\|\|\vx_\vg^k - \vx_\vf^k\|\|\vx\| \\
&\stackrel{\eqref{eq:independentFPR}}{\leq}  \frac{\left(\|\vz^0 - \vz^\ast\|_U + \|\vz^\ast - \vx\|_U\right) \|\vz^0 - \vz^\ast\|_U}{\gamma\sqrt{\underline{\tau}(k+1)}} \\
&\stackrel{\eqref{eq:FPRrates}}{+} \frac{(L + |w|\|\vS\| \|\vx\|)\|\vz^0 - \vz^\ast\|_U}{\sqrt{\rho\underline{\tau}(k+1)}}.\qquad \endproof
\end{align*}

\begin{remark}
Note that we can immediately strengthen the convergence result for PRS in Theorems~\ref{thm:nonergodic} and~\ref{thm:PDerg}. Indeed, we only need to assume that $\vf$ or $\vg$ is Lipschitz on the closed ball $\overline{B_U(\vx^\ast; \|\vz^0 - \vz^\ast\|_U)}$ (where $\vx^\ast = J_{\gamma U^{-1}(\partial \vg + (1-w)\vS)}(\vz^\ast)$) of radius $\|\vz^0 - \vz^\ast\|_U$ (under the metric $\|\cdot\|_U$) because for all $k \in \vN$,
\begin{align*}
\|\vx_\vg^k - \vx^\ast\|_U = \|J_{\gamma U^{-1}(\partial \vg + (1-w)\vS)} (\vz^k) - J_{\gamma U^{-1}(\partial \vg + (1-w)\vS)}(\vz^\ast)\|_U &\leq \|\vz^k - \vz^\ast\|_U \\&\leq \|\vz^0 - \vz^\ast\|_U,
\end{align*}
and by a similar derivation, $\| \vx_\vf^k - \vx^\ast\|_U \leq \|\vz^0 - \vz^\ast\|_U$. Thus, the sequences lie in the ball: $(\vx_\vf^j)_{j \in \vN}, (\vx_\vg^j)_{j \in \vN} \subseteq \overline{B_U(\vx^\ast, \|\vz^0 - \vz^\ast\|_U)}$. We also have $(\overline{\vx}_\vf^j)_{j \in \vN}, (\overline{\vx}_\vg^j)_{j \in \vN} \subseteq \overline{B_U(\vx^\ast, \|\vz^0 - \vz^\ast\|_U)}$ by the convexity of the ball. See \cite[Proposition 8.28]{bauschke2011convex} for conditions that ensure Lipschitz continuity of convex functions on balls.
\end{remark}

\begin{remark}
In general, the $o(1/\sqrt{k+1})$ convergence rates in Theorem~\ref{thm:nonergodic} are the best PRS can obtain for $({\vx}_\vg^j)_{j \in \vN}$ \cite[Theorem 11]{davis2014convergence}.
\end{remark}

\begin{remark}
In general, it is infeasible to take the supremum over the last component of $\cG^{\mathrm{pre}}$ as in Equation~\eqref{eq:ppdgsup}. Thus, in practice we cannot use the pre-primal-dual gap to measure convergence. However, Theorem~\ref{thm:nonergodic} bounds the pre-primal-dual gap at the $k$-th iteration by a multiple of the expression $\|T\vz^k - \vz^k\|\|\vx\|$.  Thus, if the supremum in Equation~\eqref{eq:ppdgsup} can be restricted to a bounded set $D$, then $\|T\vz^k - \vz^k\| \sup_{\vx \in D} \|\vx\|$ can be used as a proxy for the size of the pre-primal-dual gap. See section~\ref{sec:preprimaldualprimal} for examples of such sets $D$.
\end{remark}

\section{Applications}\label{sec:applications}

In this section we will show that the four algorithms from Section~\ref{sec:examplesofUS} are capable of solving highly structured optimization problems:

\begin{problem}[Model problem]\label{prop:modelminimization}
Let $\cH_0$ be a Hilbert space, and let $f, g : \Gamma_0(\cH_0)$. Let $n \in \vN\backslash \{0\}$, and for $i = 1, \cdots, n$, let $\cH_i$ be a Hilbert space, let $h_i, l_i \in \Gamma_0(\cH_i)$, suppose that $h_i \square l_i \in \Gamma_0(\cH_i)$, and let $B_i : \cH_0 \rightarrow \cH_i$ be a bounded linear map. Finally, let $\vB : \cH_0 \rightarrow \prod_{i=1}^n \cH_i$ be the map $x \mapsto (B_1x, \cdots, B_nx)$. Then our model problem is as follows:
\begin{align}\label{eq:pdproblemtosplit}
\Min_{x \in \cH_0} f(x) + g(x) + \sum_{i=1}^n (h_i \square l_i)(B_i x).
\end{align}
In addition, the dual problem is to 
\begin{align*}
\Min_{\vy \in \prod_{i=1}^n \cH_i} \; (f^\ast\square g^\ast)(-\vB^\ast \vy) + \sum_{i=1}^n (h_i^\ast + l_i^\ast)(y_i).
\end{align*}
\end{problem}
\indent All of the algorithms we consider take full advantage of the structure of the infimal convolution in Problem~\ref{prop:modelminimization}.  We note that infimal convolutions are not widespread in applications. Generally, for $i \in \{1, \ldots, n\}$, we think of $h_i \square l_i$ as a regularization of $h_i$ by $l_i$, or vice versa. Indeed, under mild conditions, the smoothness of at least one of $h_i$ and $l_i$ implies the smoothness of the infimal convolution \cite[Section 18.3]{bauschke2011convex}.  When $l_i$ or $h_i$ is chosen properly, this operation is sometimes called \emph{dual-smoothing} \cite{nesterov2005smooth}. 
Finally, we note that we can remove the infimal convolution operation from Problem~\ref{prop:modelminimization} by setting $l_i = \iota_{\{0\}}$ because $h_i \square l_i = h_i$ for all $i = 1, \cdots, n$.  The interested reader should consult~\cite[Proposition 12.14 and Proposition 15.7]{bauschke2011convex} for conditions that guarantee that $h_i \square l_i \in \Gamma_0(\cH_i)$.

We assume the existence of a specific type of solution of Problem~\ref{prop:modelminimization}.
\begin{assump}\label{assump:primaldualassump}
We assume that there exists 
\begin{align*}
x^\ast \in \zer\left( \partial f + \partial g + \sum_{i=1}^n B_i^\ast (\partial h_i \square \partial l_i)(B_i(\cdot))\right). 
\end{align*}
\end{assump}
See \cite[Proposition 4.3]{combettes2012primal} for conditions that guarantee the existence of $x^\ast$. In general, the containment
\begin{align*}
\zer\left( \partial f + \partial g +  \sum_{i=1}^n B_i^\ast (\partial h_i \square \partial l_i)( B_i(\cdot))\right) \subseteq \zer\left( \partial\left( f +  g + \sum_{i=1}^n  (h_i \square l_i)( B_i(\cdot)) \right)\right)
\end{align*}
always holds, but the sets may not be equal. Nevertheless, this assumption is standard.

We now review two possible splittings of Problem~\ref{prop:modelminimization}.  Both splittings will be designated by a ``level." The level is an indication of the number of extra dual variables that are introduced into the problem. Introducing more dual variables makes the problem further separable, and, hence, further parallelizable, but it also increases the memory footprint of the algorithm. It is unclear whether the number of dual variables affects the practical convergence speed of the algorithm in a negative way.

The following proposition is a simple exercise in duality, so we omit the proof.
\begin{proposition}[Level 1 optimality conditions]\label{prop:splittinglevel1}
Let $\vH = \prod_{i=0}^n \cH_i$, and denote an arbitrary point $\vx \in \vH$ by $\vx = (x, y_1, \cdots, y_n) = (x, \vy)$.  For all $\vx \in \vH$, let $\vf(\vx) := f(x) + \sum_{i=1}^n h_i^\ast(y_i)$, let $\vg(\vx) := g(x) + \sum_{i=1}^n l_i^\ast(y_i)$, and let $\vS : \vH \rightarrow \vH$ be the skew map $(x, \vy) \mapsto (\vB^\ast \vy, -\vB x)$.
Then a point $x^\ast \in \cH_0$ satisfies
\begin{align}
0 \in \partial f(x^\ast) + \partial g(x^\ast) + \sum_{i=1}^n  B_i^\ast(\partial h_i \square \partial l_i)(B_i x^\ast)
 \end{align}
if, and only if, there is a vector $\vy^\ast \in \prod_{i=1}^n\cH_i$ such that 
\begin{align}\label{eq:leveloneinclusion}
0 \in \partial \vf(x^\ast, \vy^\ast) + \partial \vg(x^\ast, \vy^\ast) + \vS(x^\ast, \vy^\ast).
\end{align}
\end{proposition}
\indent Notice that the subdifferential operators $\partial \vf$ an $\partial \vg$ in Equation~\eqref{eq:leveloneinclusion} are completely separable in the variables of the product space $\vH$.  Thus, evaluating the proximity operators of $\vf$ and $\vg$ can be quite simple.  However, the resolvent $J_{\partial \vf + \vS}$ is not necessarily simple to evaluate. This difficulty motivates the introduction of new metrics on $\vH$ that simplify the resolvent computation (Section~\ref{sec:algorithmclasses}).

Whenever the functions $g$ and $l_i^\ast$ are Lipschitz differentiable for $i \in \{1, \ldots, n\}$ (or equivalently, $l_i$ is strongly convex \cite[Theorem 18.15]{bauschke2011convex}) we can apply FBS or FBF (Algorithms~\ref{alg:PDFBS} and~\ref{alg:PDFBF}) to the splitting in Proposition~\ref{prop:splittinglevel1}. For nonsmooth $g$ and $l_i^\ast$, we can apply the PRS algorithm.

The proof of the following proposition is similar to Proposition~\ref{prop:splittinglevel1}, so we omit it. The proposition is most useful in the case that $g$ or $l_i^\ast$ are not differentiable for some $i \in \{1, \ldots, n\}$.

\begin{proposition}[Level 2 optimality conditions]\label{prop:splittinglevel2}
Let $\vH = \cH_0 \times (\prod_{i=1}^n \cH_i)^2$, and denote an arbitrary $\vx \in \vH$ by $\vx = (x, y_1, \cdots, y_n, v_1, \cdots, v_n) = (x, \vy, \vv)$. For all $\vx \in \vH$, let $\vf(\vx) := f(x) + \sum_{i=1}^n (h_i^\ast(y_i) +l_i(v_i))$, let $\vg(\vx) := g(x)$, and let $\vS : \vH \rightarrow \vH$ be the skew map $(x, \vy, \vv) \mapsto (\vB^\ast \vy, -\vB x + \vv,  -\vy)$. 
Then a point $x^\ast \in \cH_0$ satisfies
\begin{align}
0 \in \partial f(x^\ast) + \partial g(x^\ast) + \sum_{i=1}^n  B_i^\ast(\partial h_i \square \partial l_i)(B_i x^\ast)
 \end{align}
if, and only if, there is a vector $(\vy^\ast, \vv^\ast) \in (\prod_{i=1}^n\cH_i)^2$ such that 
\begin{align}\label{eq:leveltwoinclusion}
0 \in \partial \vf(x^\ast, \vy^\ast, \vv^\ast) + \partial \vg(x^\ast, \vy^\ast, \vv^\ast) + \vS(x^\ast, \vy^\ast,\vv^\ast).
\end{align}
\end{proposition}

Note that if for some $i \in \{1, \ldots, n\}$, $l_i$ is differentiable, we can ``assign" it to the function $\vg$ instead of ``assigning" it to $\vf$.  If $g$ is also differentiable, we can apply FBS to the inclusion.

There are many splittings that solve Problem~\ref{prop:modelminimization}. Furthermore, the complexity of Problem~\ref{prop:modelminimization} can be increased in various ways, e.g., by precomposing each of $h_i$ and $l_i$ with linear operators \cite{becker2013algorithm,bot2013algorithm}, or by solving systems of such inclusions \cite{combettes2013systems,boct2013solving}.  We choose to discuss this relatively simple formulation for clarity of exposition.

The next several sections relate the results and notation of the previous sections to the level 1 and 2 splittings.  

\subsection{Primal-dual gap functions}\label{sec:preprimaldualprimal}

In this section, we discuss the pre-primal-dual gap function in the context of the level 1 splitting in Proposition~\ref{prop:splittinglevel1}.  We give sufficient conditions for the gap function (Definition~\ref{defi:preprimaldualgap}) to bound the primal and dual objectives of Problem~\ref{prop:modelminimization} and show that the pre-primal-dual gap also bounds certain squared norms that arise from the strong convexity and differentiability of the terms of the objective.

In the level 1 splitting, the pre-primal-dual gap has the following form: for all $(x, \vy), (x^\ast,\vy^\ast) \in \vH$ (with components defined as in Proposition~\ref{prop:splittinglevel1}), we have
\begin{align*}
 \cG^{\mathrm{pre}}(\vx,\vx, \vx; \vx^\ast) &= f(x) + g(x) - f(x^\ast) - g(x^\ast) + \dotp{x-x^\ast, \vB^\ast \vy^\ast} \\
&+ \sum_{i=1}^n \left(h_i^\ast(y_i) + l_i^\ast(y_i) - h_i^\ast(y^\ast_i) - l_i^\ast(y_i^\ast)\right) - \dotp{\vB x^\ast, \vy - \vy^\ast}, \numberthis \label{eq:ppdgaplevel1}
\end{align*}
where we used the identity $\dotp{\vS \vx, -\vx^\ast} = \dotp{\vS\vx, \vx - \vx^\ast}$. If $\vx^\ast$ satisfies the inclusion in Proposition~\ref{prop:splittinglevel1}, then
\begin{align}\label{eq:ppdgaplevel1subs}
-\vB^\ast \vy^\ast \in  \partial f(x^\ast) + \partial g(x^\ast) && \mathrm{and} && B_i x^\ast \in \partial h_i^\ast(y_i^\ast) + \partial l_i^\ast(y_i^\ast).
\end{align}
We will now bound several terms that arise from the strong convexity and Lipschitz differentiability of the terms in the objective function.

We follow the convention that every closed, proper, and convex function $F : \cH_0 \rightarrow (-\infty, \infty]$ is $\mu_F$-strongly convex and $\tnabla F$ is $L_F$-Lipschitz for some $\mu_F \in \vR_+$ and $L_F  \in [0, +\infty]$. If $F$ is not differentiable, then we let $L_F = \infty$. In addition, if $L_F < \infty$, then $\tnabla F = \nabla F$ is Lipschitz. Note that we allow the $\mu_F = 0$.  The following quantity is useful for summarizing the lower bounds that we derive from strong convexity and Lipschitz differentiability: for all $x \in \cH_0$ and $y \in \dom(\partial F)$, if
\begin{align}\label{eq:snotation}
S_F(x, y) &:=  \begin{cases}
\max\left\{\frac{\mu_F}{2}\|x - y\|^2, \frac{1}{2L_F}\|\nabla F(x) - \nabla F(y)\|^2\right\} & \text{if } L_F < \infty; \\
\frac{\mu_F}{2}\|x - y\|^2 & \text{otherwise;}
\end{cases}
\end{align}
then combine \cite[Theorem 18.15(iv) and Proposition 16.9]{bauschke2011convex} to get 
\begin{align}\label{eq:strongconvexandlipschitzlowerbound}
F(x) &\geq F(y) + \dotp{x-y, \tnabla F(y)} + S_F(x, y). 
\end{align}
We use the analogous notation for $f, g$ and the conjugate functions $h_i^\ast, l_i^\ast$ for $i=1, \cdots, n$.  Therefore, if we apply the lower bound in Equation~\eqref{eq:strongconvexandlipschitzlowerbound} to each of the functions in Equation~\eqref{eq:ppdgaplevel1} and use the subgradient identities in Equation~\eqref{eq:ppdgaplevel1subs} to cancel inner products, we get
\begin{align}\label{eq:primaldualgapboundstrongterms}
 \cG^{\mathrm{pre}}(\vx,\vx, \vx; \vx^\ast) &\geq S_f(x, x^\ast) + S_g(x, x^\ast) + \sum_{i=1}^n \left(S_{h_i^\ast}(y_i, y_i^\ast) + S_{l_i^\ast}(y_i, y_i^\ast)\right).
\end{align}
Equation~\eqref{eq:primaldualgapboundstrongterms} shows that convergence rates for the pre-primal-dual gap function immediately imply the same convergence rates for the $S_{\cdot}(\cdot, \cdot)$ functions in Equation~\eqref{eq:snotation}.  Note that this lower bound does not require that $\dom(\vf)$ or $\dom(\vg)$ are bounded.  

The next proposition gives sufficient conditions under which the pre-primal-dual gap bounds the primal and dual objectives.  In general, we cannot expect such a bound to hold, unless several terms in the objective are Lipschitz continuous or certain subdifferentials are locally bounded.

\begin{proposition}[Level 1 gap function bounds]\label{prop:gapsplitting1}
Let $x^\ast$ be a minimizer of Problem~\ref{prop:modelminimization}. Assume the notation of Proposition~\ref{prop:splittinglevel1}. Let $D_1 \subseteq \cH$ and let $D_2 \subseteq \prod_{i=1}^n \cH_i$ be bounded sets. Then for any sequence of points $((x^j, \vy^j))_{j\geq0} \subseteq \dom(f + g)\times \prod_{i = 1}^n \dom(h_i^\ast + l_i^\ast)$, the inequality
\begin{align*}
& f(x^k) + g(x^k) + \sum_{i=1}^n (h_i \square l_i)(B_i x^k) -  \left(f(x^\ast) + g(x^\ast) + \sum_{i=1}^n (h_i \square l_i)(B_i x^\ast)\right) \\
&\leq \sup_{\vx \in \{x^\ast\}\times D_2} \cG^{\mathrm{pre}}(\vx^k,\vx^k, \vx^k; \vx)
\end{align*}
holds for all $k \in \vN$ provided either of the following hold:
\begin{remunerate}
\item\label{prop:gapsplitting1:partdomain} $\dom(h_1^\ast + l_1^\ast) \times \cdots \times \dom(h_n^\ast + l_n^\ast) \subseteq D_2$;
\item \label{prop:gapsplitting1:part:subgradient}  $\partial (h_1 \square l_1)(B_1x^k) \times \cdots \times  \partial (h_n \square l_n)(B_nx^k) \subseteq D_2$.
\end{remunerate}

Similarly, the inequality
\begin{align*}
 &(f^\ast \square g^\ast)(-\vB^\ast \vy^k) + \sum_{i=1}^n (h_i^\ast + l_i^\ast)(y_i^k) -  \left((f^\ast\square g^\ast)(-\vB^\ast \vy^\ast) + \sum_{i=1}^n (h_i^\ast + l_i^\ast)(B_i y_i^\ast)\right) \\
 &\leq \sup_{\vx \in D_1\times \{\vy^\ast\}} \cG^{\mathrm{pre}}(\vx^k,\vx^k, \vx^k; \vx)
\end{align*}
holds for all $k \in \vN$ provided either of the following hold:
\begin{remunerate}
\item\label{prop:gapsplittingdual:partdomain}  $\dom(f + g) \subseteq D_1$;
\item \label{prop:gapsplittingdual:part:subgradient}  $\partial (f^\ast \square g^\ast)(-\vB^\ast \vy^k) \subseteq D_1$.
\end{remunerate}
\end{proposition}
{\em Proof.} 
Fix $k \in \vN$. We only consider the primal case because the dual case is similar. For all $i \in \{1, \cdots, n\}$, the Fenchel-Moreau Theorem~\cite[Theorem 13.32]{bauschke2011convex}, the identity $h_i\square l_i = (h_i^\ast + l_i^\ast)^\ast$, and Conditions~\ref{prop:gapsplitting1:partdomain} and~\ref{prop:gapsplitting1:part:subgradient} show that we can reduce the domain of the following supremum: 
\begin{align*}
\sum_{i=1}^n (h_i\square l_i)(B_{i}x^k) &= \sup_{\vy\in \vH} \left(\dotp{\vB x^k, \vy} - \sum_{i=1}^n(h_i^\ast(y_i) + l_i^\ast(y_i))\right) \\
&= \sup_{\vy \in D_2}\left( \dotp{\vB x^k, \vy} - \sum_{i = 1}^n(h_i^\ast(y_i) + l_i^\ast(y_i))\right).
\end{align*}
In addition, the Fenchel-Young inequality shows that 
\begin{align*}
\sum_{i =1}^n \left(h_i^\ast(y_i^k) + l_i^\ast(y_i^k)\right) - \dotp{x^\ast, \vB^\ast\vy^k}& \geq  -  \sum_{i = 1}^n (h_i \square l_i)(B_ix^\ast).
\end{align*}
Therefore, 
\begin{align*}
& \sup_{\vx \in \{x^\ast\} \times D_2} \cG^{\mathrm{pre}}(\vx^k,\vx^k, \vx^k; \vx) \\
&= f(x^k) + g(x^k) - f(x^\ast) - g(x^\ast) + \sum_{i=1}^n (h_i^\ast(y_i^k) + l_i^\ast(y_i^k)) - \dotp{x^\ast, \vB^\ast\vy^k} \\
&\hspace{20pt}+\sup_{\vy \in D_2} \left(\dotp{\vB x^k, \vy} - \sum_{i = 1}^n(h_i^\ast(y_i) + l_i^\ast(y_i))\right) \\
&\geq  f(x^k) + g(x^k) + \sum_{i=1}^n (h_i \square l_i)(B_i x^k) -  \left(f(x^\ast) + g(x^\ast) + \sum_{i=1}^n (h_i \square l_i)(B_i x^\ast)\right). \qquad\endproof
\end{align*}

Fix $i \in \{1, \ldots, n\}$. The bounded domain conditions in Proposition~\ref{prop:gapsplitting1} are related to the Lipschitz continuity of the objective functions.  Indeed, if $h_i$ is Lipschitz, it follows that $\dom(h_i^\ast)$ is bounded \cite[Proposition 4.4.6]{borwein2010convex}.  In addition, $\dom(h_i^\ast+ l_i^\ast) = \dom(h_i^\ast) \cap \dom(l_i^\ast)$. Thus, if $h_i^\ast$ has bounded domain, so does $h_i^\ast+ l_i^\ast$. 

The bounded subgradient conditions in Proposition~\ref{prop:gapsplitting1} are satisfied for $h_i\square l_i$ if the infimal convolution is continuous everywhere and the sequence $(B_ix^j)_{j \in \vN}$ is convergent. Indeed, in this case $\partial (h_i\square l_i)$ is locally bounded \cite[Proposition 16.14(iii)]{bauschke2011convex} and hence, the union $\bigcup_{j \in \vN} \partial (h_i \square l_i)(B_i x^j)$ is bounded. See~\cite[Remark 2.2]{boct2014convergence} and~\cite{bo2014convergence} for similar remarks in the context of primal-dual FBF and FBS algorithms.

\subsection{Two algorithm classes}\label{sec:algorithmclasses}

In this section, we study the algorithms that arise for different classes of maps $(U_j)_{j \in \vN}$ and show how to compute the resolvent and forward-backward operators needed in order to apply the PPA, FBS, PRS, and FBF algorithms just as they appear in Section~\ref{sec:US}. 

We fix the following notation for the rest of this section: Let $\mu_{V_i} > 0$ and let $V_i \in \cS_{\mu_{V_i}}(\cH_i)$ for $i = 0, \cdots, n$. Let $\mu_{W_i} > 0$ and let  $W_i \in \cS_{\mu_{W_i}}(\cH_i)$ for $i =1, \cdots, n$.  These strongly monotone maps induce metrics on the spaces $\cH_i$ for $i=0, \cdots, n$. They can be as simple as ``diagonal" metrics, but they can also incorporate second order information. A discussion on the best metric choice is beyond the scope of this paper, so we just refer the reader to \cite{pock2011diagonal} for some applications of  fixed ``diagonal" metrics, and \cite{goldstein2013adaptive} for varying ``diagonal" metrics that satisfy conditions akin to Assumption~\ref{assump:variablemetric}.

Now define ``block-diagonal" maps
\begin{align}\label{eq:blockdiagonal}
\vV := V_1 \oplus \cdots \oplus V_n  \in \cS_{\mu_{\vV}}\left(\prod_{i=1}^n\cH_i\right) && \mathrm{and} && \vW :=  W_1 \oplus \cdots \oplus W_n \in\cS_{\mu_{\vW}}\left(\prod_{i=1}^n\cH_i\right)
\end{align}
where $\mu_{\vV} = \min\{\mu_{V_1}, \cdots, \mu_{V_n}\}$, and $\mu_{\vW} = \min\{\mu_{W_1}, \cdots, \mu_{W_n}\}$.
The rest of this section will build three types of metrics from $V_0, \vV, \vW$. 

Finally, note that Part~\ref{prop:basicprox:part:J} of Proposition~\ref{prop:basicprox} shows the following: for all $\vz \in \vH$, 
\begin{align}\label{eq:equivalencemetricresolvent}
\vz^+ = J_{U^{-1} (\partial \vf + \vS)}(\vz) && \Longleftrightarrow && U(\vz - \vz^+) \in \partial \vf(\vz^+) + \vS\vz^+.
\end{align}
See Proposition~\ref{prop:metricstdFBS},~\ref{prop:metric2FBS}, and~\ref{prop:metricw0FBS} for examples of resolvent computations.

\subsubsection{First metric class}\label{sec:firstmetrics}

In this section, our metrics depend on a parameter $w$, which appears in  Algorithm~\ref{alg:PDPRS}.  We only use the metric for the case that $w \in \{0, 1/2, 1\}$, but we state all of our results for the general case $w \in \vR$.  The case $w = 1/2$ first appeared in \cite[Theorem 2.1]{bo?2013douglas} (for certain $\vV$ and $V_0$), and the case $w = 1$ first appeared in \cite[Equation (2.5)]{he2012convergence} (for certain $\vV$ and $V_0$). See also~\cite[Relation (3.14)]{vu2013splitting}.

\begin{proposition}\label{prop:metricclass1}
Let $w \in \vR$. Assume the setting of Proposition~\ref{prop:splittinglevel1}. Define a map $U_w : \vH \rightarrow \vH$ as follows: for all $\vx = (x, \vy) \in \vH$,
\begin{align}\label{eq:level1metric}
U_w\vx &:= \left(V_0x - w\vB^\ast\vy, -w\vB x + \vV\vy  \right).
\end{align}
Suppose that $w^2\|\vV^{-1/2} \vB V_0^{-1/2}\|^2 < 1$. Then $U_w$ is self adjoint and strongly monotone: for all $\vx \in \vH$,
\begin{align}\label{eq:level1monotone}
\dotp{\vx, U_w\vx} &\geq \frac{1}{2}\left(1 - w^2\|\vV^{-1/2} \vB V_0^{-1/2}\|^2\right) \min\{\mu_{V_0}, \mu_{\vV}\}\left( \|x\|^2 + \|\vy\|^2\right).
\end{align}

Assume the setting of Proposition~\ref{prop:splittinglevel2}. Define a map $U_w' : \vH \rightarrow \vH$ as follows: for all $\vx = (x, \vv, \vy) \in \vH$,
\begin{align}\label{eq:level2metric}
U_w'\vx &:= \left(V_0x - w\vB^\ast\vy , \vV\vy -w\vB x   + w\vv, w\vy + \vW \vv  \right).
\end{align}
Suppose that $w^2\|\vV^{-1/2} \vB V_0^{-1/2}\|^2+ w^2\|\vW^{-1/2} \vV^{-1/2}\|^2 < 1$. Then 
\begin{align*}
\dotp{\vx, U_w'\vx} &\geq \frac{1}{3}\left(1 - w^2\|\vV^{-1/2} \vB V_0^{-1/2}\|^2- w^2\|\vW^{-1/2} \vV^{-1/2}\|^2\right) \\
&\times \min\{\mu_{V_0},\mu_{\vV}, \mu_{\vW} \}\left( \|x\|^2 + \|\vy\|^2 + \|\vv\|^2\right). \numberthis\label{eq:level2monotone}
\end{align*}
\end{proposition}
\indent We omit the proof of Proposition~\ref{prop:metricclass1} because Equation~\eqref{eq:level1monotone} is shown in~\cite[Lemma 4.3, Equation (4.14)]{pesquet2014class} when $w = 1$, the extension to general $w$ is straightforward, and Equation~\eqref{eq:level2monotone} has nearly the same proof. 

Note that our conditions for ergodic convergence in Theorem~\ref{thm:PDerg} require the metric inducing maps to be almost decreasing up to a summable residual in the Loewner partial ordering $\succcurlyeq$ (see Section~\ref{sec:notation}).  If $w \in \vR$ and $((U_w)_j)_{j \in \vN}$ is a sequence of maps defined as in Equation~\eqref{eq:level1metric}, we have
\begin{align*}
((U_w)_k - (U_w)_{k+1})\vx = \left(\left(V_{0, k} - V_{0, k+1}\right)x, \left(\vV_k - \vV_{k+1}\right)\vy\right)
\end{align*}
for all $\vx \in \vH$ and $k \in \vN$. Thus, if for all $k \in \vN$, we have $V_{0,k} \succcurlyeq V_{0, k+1}$ and $\vV_{k} \succcurlyeq \vV_{k+1}$, we can guarantee that the product metric is decreasing (Lemma~\ref{lem:metricproperties}).  A similar result holds for the level 2 metrics in Equation~\eqref{eq:level2metric}.  

The following proposition shows how to evaluate the FBS operator under the metrics induced by $U_w$ and $U_w'$. Note that the results of Proposition~\ref{prop:metricstdFBS} are not new. The level 1 case with $w \in \{0, 1/2, 1\}$ has appeared implicitly in several papers, including \cite{condat2013primal,vu2013splitting,combettes2012variable}. It has also explicitly appeared in~\cite[Lemma 4.5]{pesquet2014class}. In addition, the proof of the level 2 case appeared in \cite[Equation (2.38)]{bo?2013douglas}. Thus, we omit the proof. 

\begin{proposition}[Forward-Backward operators under the first metric class]\label{prop:metricstdFBS}
Let $w \in \vR$. Assume the setting of Propositions~\ref{prop:splittinglevel1} and~\ref{prop:metricclass1}, and suppose that $U_w \in \cS_\rho(\vH)$ (Equation~\eqref{eq:level1metric}) for some $\rho> 0$. Let $\vz := (x, \vy) \in \vH$. Suppose that $g, l_1^\ast, \cdots, l_n^\ast$ are differentiable. Then $\vz^+ := J_{U_w^{-1} \left(\partial \vf + w\vS\right)}(\vz - U_w^{-1}\nabla \vg(\vz))$ has the following form: $\vz^+ = (x^+, \vy^+) \in \vH$ where 

\RestyleAlgo{plain}
\SetAlgoVlined
\begin{algorithm}[H]
\SetKwInOut{Input}{input}\SetKwInOut{Output}{output}
\SetKwComment{Comment}{}{}
{
$x^+ = \prox_{f}^{V_0}(x - V_0^{-1} (w\vB^\ast \vy + \nabla g(x)))$\;
\For{$i=1,~2,\ldots, n$, in parallel}{
$y_{i}^{+} = \prox_{h_i^\ast}^{V_i} (y_i + V_i^{-1}(wB_i(2x^{+} - x) - \nabla l_i^\ast (y_i)))$\;
}}
\label{alg:levelonefboperator}
\end{algorithm}

Assume the setting of Proposition~\ref{prop:splittinglevel2}, and suppose that $U_w' \in \cS_\rho(\vH)$ (Equation~\eqref{eq:level2metric}) for some $\rho> 0$. Let $\vz := (x, \vy, \vv) \in \vH$, and suppose that $g$ is differentiable. Then $\vz^+ := J_{(U_w')^{-1} \left(\partial \vf + w\vS\right)}(\vz - (U_w')^{-1}\nabla \vg(\vz))$ has the following form: $\vz^+ = (x^+, \vv^+, \vy^+) \in \vH$ where 

\RestyleAlgo{plain}
\SetAlgoVlined
\begin{algorithm}[H]
\SetKwInOut{Input}{input}\SetKwInOut{Output}{output}
\SetKwComment{Comment}{}{}
{
$x^+ = \prox_{f}^{V_0}(x - V_0^{-1} (w\vB^\ast \vy + \nabla g(x)))$\;
\For{$i=1,~2,\ldots, n$, in parallel}{
$v_i^+ = \prox_{l_i}^{W_i}(v_i + wW_i^{-1} y_i)$\;
$y_{i}^{+} = \prox_{h_i^\ast}^{V_i} (y_i + V_i^{-1}(wB_i(2x^{+} - x) - (2v_i^+ - v_i)))$\;
}}
\label{alg:leveltwofboperator}
\end{algorithm}
\end{proposition}

\subsection{Second metric class}

The following result is similar to~\cite[Lemma 4.9]{pesquet2014class} (which applies to $(U_w)^{-1}$). 
\begin{proposition}\label{prop:metricclass2}
Assume the setting of Proposition~\ref{prop:splittinglevel1}. Define a map $U_w : \vH \rightarrow \vH$ as follows: for all $\vx = (x, \vy) \in \vH$,
\begin{align}\label{eq:level1metric2}
U_w\vx &:= \left(V_0x, (\vV- w^2\vB V_0^{-1}\vB^\ast)\vy  \right).
\end{align}
Suppose that $w^2\|\vV^{-1/2} \vB V_0^{-1/2}\|^2 < 1$. Then $U_w$ is self adjoint and strongly monotone: for all $\vx \in \vH$,
\begin{align}\label{eq:level1monotone2}
\dotp{\vx, U_w\vx} &\geq  \min\left\{\mu_{V_0}, \left(1 - w^2\|\vV^{-1/2} \vB V_0^{-1/2}\|^2\right)\mu_{\vV} \right\}\left( \|x\|^2 + \|\vy\|^2\right).
\end{align}
\end{proposition}
{\em Proof.}
Set $\vC = w\vB$. For all $\vy \in \prod_{i=1}^n \cH_i$, we have
\begin{align*}
\dotp{\vy, (\vV - \vC V_0^{-1} \vC^\ast)\vy}  &= \dotp{\vV^{1/2}\vy,\left( I_{\prod_{i=1}^n \cH_i} - \vV^{-1/2}\vC V_0^{-1} \vC^\ast\vV^{-1/2}\right)\vV^{1/2}\vy} \\
&= \dotp{ \vV\vy, \vy} - \dotp{\vV^{1/2}\vy,\vV^{-1/2}\vC V_0^{-1} \vC^\ast \vV^{-1/2}\vV^{1/2}\vy} \\
&\geq \left( 1- \|\vV^{-1/2}\vC^\ast V_0^{-1} \vC \vV^{-1/2}\|\right)\dotp{ \vV\vy, \vy} \\
&\geq \left(1 - w^2\|\vV^{-1/2} \vB V_0^{-1/2}\|^2\right)\mu_{\vV}\|\vy\|^2.
\end{align*}
Therefore, 
\begin{align*}
\dotp{\vx,U_w\vx} &\geq \mu_{V_0}\|x\|^2 + \left(1 - w^2\|\vV^{-1/2} \vB V_0^{-1/2}\|^2\right)\mu_{\vV}\|\vy\|^2 \\ 
&\geq \min\left\{\mu_{V_0}, \left(1 - w^2\|\vV^{-1/2} \vB V_0^{-1/2}\|^2\right)\mu_{\vV} \right\}\left( \|x\|^2 + \|\vy\|^2\right).\qquad\endproof
\end{align*}

For simplicity and because it has not yet found an application we do not discuss the generalization of the Equation~\eqref{eq:level1metric2} to the level 2 case. 

Note that our conditions for ergodic convergence in Theorem~\ref{thm:PDerg} require the metric inducing maps to be almost decreasing, up to a summable residual, in the Loewner partial ordering $\succcurlyeq$ (see Section~\ref{sec:notation}).  If $w \in \vR$ and $((U_w)_j)_{j \in \vN}$ is a sequence of maps defined as in Equation~\eqref{eq:level1metric2}, we have
\begin{align*}
((U_w)_k - (U_w)_{k+1})\vx = \left(\left(V_{0, k} - V_{0, k+1}\right)x, \left((\vV_k - \vV_{k+1}) + w^2\vB (V_{0, k+1}^{-1} - V_{0, k}^{-1})\vB^\ast)\right)\vy\right)
\end{align*}
for all $\vx \in \vH$ and $k \in \vN$. Thus, if for all $k \in \vN$, we have $V_{0, k} \succcurlyeq V_{0, k+1}$ and $\vV_{k} \succcurlyeq \vV_{k+1}$, the product metric is decreasing (Lemma~\ref{lem:metricproperties}).

The following proposition shows how to evaluate the FBS operator under the metric induced by $U$. Note that Proposition~\ref{prop:metric2FBS} appears in \cite[Lemma 4.10]{pesquet2014class} for $w = 1$. Thus, we omit the proof. 

\begin{proposition}[Forward-Backward operators under the second metric class]\label{prop:metric2FBS}
Assume the setting of Proposition~\ref{prop:splittinglevel1}. Suppose that $f \equiv 0$, and that $U \in \cS_\rho(\vH)$ (Equation~\eqref{eq:level1metric2}) for some $\rho> 0$. Let $\vz := (x, \vy) \in \vH$. Suppose that $g, l_1^\ast, \cdots, l_n^\ast$ are differentiable. Then $\vz^+ := J_{U^{-1} \left(\partial \vf + \vS\right)}(\vz - U^{-1}\nabla \vg(\vz))$ has the following form: $\vz^+ = (x^+, \vy^+) \in \vH$ where 

\RestyleAlgo{plain}
\SetAlgoVlined
\begin{algorithm}[H]
\SetKwInOut{Input}{input}\SetKwInOut{Output}{output}
\SetKwComment{Comment}{}{}
{
\For{$i=1,~2,\ldots, n$, in parallel}{
$y_{i}^{+} = \prox_{h_i^\ast}^{V_i} \left(y_i + V_i^{-1}\left(wB_i\left(x - V_0^{-1}(\nabla g(x) + w\vB^\ast \vy\right) -  \nabla l_i^\ast (y_i)\right)\right)$\;
}
$x^+ = x - V_0^{-1}\left( \nabla g(x) + w\vB^\ast \vy^+\right)$\;}
\label{alg:leveltwofbsoperator}
\end{algorithm}
\end{proposition}

Now consider the special case $w = 0$. In this case, the first and second metric classes agree. The following Proposition with $U = I_{\vH}$ appears in \cite[Proposition 2.7]{briceno2011monotone+}. Our generalization is straightforward, so we omit the proof. 

\begin{proposition}[Resolvents of skew operators]\label{prop:metricw0FBS}
Assume the setting of Proposition~\ref{prop:splittinglevel1}. Let $w \in \vR$ and suppose that $U_w \in \cS_\rho(\vH)$ (Equation~\eqref{eq:level1metric2}) for some $\rho> 0$. Let $\vz := (x, \vy) \in \vH$. Then $\vz^+ := J_{\gamma U^{-1} \vS}(\vz)$ has the following form: $\vz^+ = (x^+, \vy^+) \in \vH$ where
\begin{align*}
x^+ &:= (I_{\cH_0} + \gamma^2 V_0\vB^\ast \vV \vB)^{-1}(x - \gamma V_0 \vB^\ast)\vy \\
\vy^+ &:=(I_{\prod_{i=1}^n\cH_i} + \gamma^2 \vV\vB V_0 \vB^\ast)^{-1} (\vy + \gamma \vV\vB x)
\end{align*}
\end{proposition}
\indent Generalizing the resolvent operator computation in Proposition~\ref{prop:metricw0FBS} to the level 2 case is straightforward, though slightly messy. It has not found application in the literature yet, so we omit the statement.

\subsection{New and old convergence rates}\label{sec:oldandnew}

Table~\ref{tab:literature} lists the application of PPA, FBS, PRS, and FBF algorithms under the metrics introduced in Section~\ref{sec:algorithmclasses} and indicates which convergence rates have been shown in the literature.  We note that, to the best of our knowledge, for all of the methods we discuss, the nonergodic fixed metric convergence rates, the ergodic convergence rates under variable metrics, and the nonergodic/ergodic convergence rates with nonconstant relaxation have never appeared in the literature.

Any pairing between metrics, algorithms, and splittings that does not appear in Table~\ref{tab:literature} is an algorithm where, to the best of our knowledge, no convergence rate has appeared in the literature.  

\begin{center}
\begin{table}
\centering
    \begin{tabular}{llllll}
    \toprule
    Reference    & Algorithm  & Metric & Level & $w$ & Rates          \\ \toprule
    \cite[Algorithm 1]{chambolle2011first} &  PPA & \eqref{eq:level1metric} & 1 & $1$& $O(1/(k+1))$ ergodic \cite{chambolle2011first}  \\\midrule
    \cite[Algorithm 2.2]{bo?2013douglas} & PPA & \eqref{eq:level2metric} & 2 & $1$ & none\\  \midrule
    \cite{condat2013primal,vu2013splitting} & FBS & \eqref{eq:level1metric} & 1& $1$ & $O(1/(k+1))$ ergodic \cite{bo2014convergence}  \\  \midrule
    \cite{chen2013primal,combettes2014forward,pesquet2014class} & FBS & \eqref{eq:level1metric2} & 1 & $1$& none \\  \midrule
    \cite[Algorithm 2.1]{bo?2013douglas} & PRS & \eqref{eq:level1metric} & 1 & $1/2$& none \\  \midrule
    \cite[Remark 2.9]{briceno2011monotone+}, \cite{o2014primal} & PRS & \eqref{eq:level1metric2} & 1 & $0$ & none \\  \midrule
    \cite{briceno2011monotone+,combettes2012primal} & FBF & \eqref{eq:level1metric2} & 1 & $0$ & $O(1/(k+1))$ ergodic~\cite{boct2014convergence} \\  \bottomrule
    \end{tabular}
 \caption{This table lists the original appearance of the algorithms constructed from pairing the metrics in Section~\ref{sec:algorithmclasses} with the PPA, FBS, PRS, and FBF algorithms applied to Problem~\ref{prop:modelminimization}.  See Propositions~\ref{prop:splittinglevel1} and~\ref{prop:splittinglevel2} for the definitions of the ``level."}
\end{table}
\label{tab:literature}
\end{center} 

\section{Conclusion}
In this paper, we provided a convergence rate analysis of a general monotone inclusion problem under the application of four different algorithms.  We provided \emph{ergodic} convergence rates under variable metrics, stepsizes, and relaxation, and recovered several known rates in the process.  In addition, for three of the algorithms we provided the first \emph{nonergodic} primal-dual gap convergence rates that have appeared in the literature. Finally, we showed how our results imply convergence rates of a large class of primal-dual splitting algorithms.  The techniques developed in this paper are not limited to the four algorithms we chose to study, and the proofs of this paper can be used as a template for proving convergence rates of other special cases of the unifying scheme.

\section*{Acknowledgement}
We thank Professor Wotao Yin and the two anonymous referees; their comments were invaluable.


\bibliographystyle{siamNoFirst.bst}
\bibliography{100307Bibliography}


\end{document}